\documentclass[11pt]{article}
\usepackage{amsmath,amsfonts,amssymb,amsthm,mathrsfs}
\usepackage{minitoc}

\usepackage[cmintegrals,libertine]{newtxmath}
\usepackage[cal=euler]{mathalfa}
\usepackage[mono=false]{libertine}
\useosf

\usepackage[dvipsnames]{xcolor}
\usepackage[a4paper,vmargin={3.5cm,3.5cm},hmargin={2.5cm,2.5cm}]{geometry}
\usepackage[font=sf, labelfont={sf,bf}, margin=1cm]{caption}
\usepackage{graphicx,graphics}
\usepackage{epsfig}
\usepackage{latexsym}
\usepackage[applemac]{inputenc}
\usepackage{stmaryrd}
\usepackage{ae,aecompl}

\usepackage[english]{babel}
 \usepackage[colorlinks=true]{hyperref}
\usepackage{pstricks}
\usepackage{enumerate}
\usepackage[normalem]{ulem}

\renewcommand{\P}{\mathbb{P}}
\newcommand{\E}{\mathbb{E}}

\newcommand{\Cinfty}{\mathscr{C}_\infty}
\DeclareFixedFont{\beaupetit}{T1}{ftp}{b}{n}{2cm}

\newtheorem{theorem}{Theorem}[]
\newtheorem{definition}{Definition}[]
\newtheorem{proposition}[]{Proposition}
\newtheorem{open}[]{Open question}

\newtheorem{lemma}[]{Lemma}

\newtheorem{conjecture}[]{Conjecture}
\theoremstyle{definition}
\newtheorem*{remark}{Remark}

\def\build#1_#2^#3{\mathrel{
\mathop{\kern 0pt#1}\limits_{#2}^{#3}}}

\newcommand{\eps}{\varepsilon}

\def\Reff{R_{\rm eff}}

\title{Geometric and spectral properties of causal maps}
\author{Nicolas Curien, Tom Hutchcroft, and Asaf Nachmias}
             
\begin{document}
             \maketitle

             \abstract{We study the random planar map obtained from a critical, finite variance, Galton-Watson plane tree by adding the horizontal connections between successive vertices at each level. This random graph is closely related to the well-known \emph{causal dynamical triangulation} that was introduced by Ambj\o rn and Loll and has been studied extensively  by physicists. 
             We prove that the horizontal distances in the graph are smaller than the vertical distances, but only by a subpolynomial factor: The diameter of the set of vertices at level $n$ is both $o(n)$ and $n^{1-o(1)}$.
              This enables us to prove that the spectral dimension of the infinite version of the graph is almost surely equal to 2, and consequently that the random walk is diffusive almost surely. 
              We also initiate an investigation of the case in which the offspring distribution is critical and belongs  to the domain of attraction of an $\alpha$-stable law for $\alpha \in (1,2)$, for which our understanding is much less complete.}
\begin{figure}[!h]
 \begin{center}
   \includegraphics[width=5cm,angle=90]{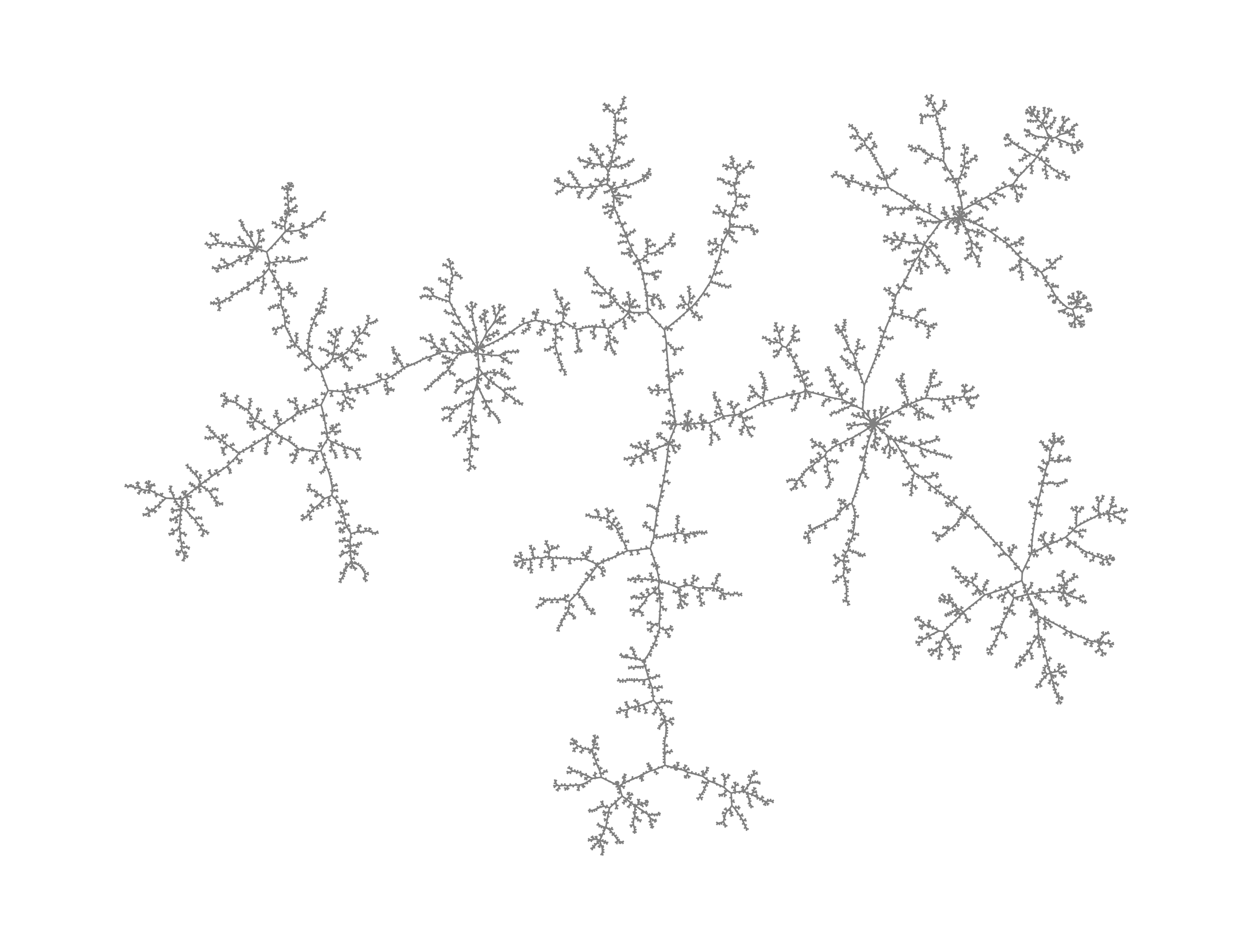}
 \includegraphics[width=5cm,angle=90]{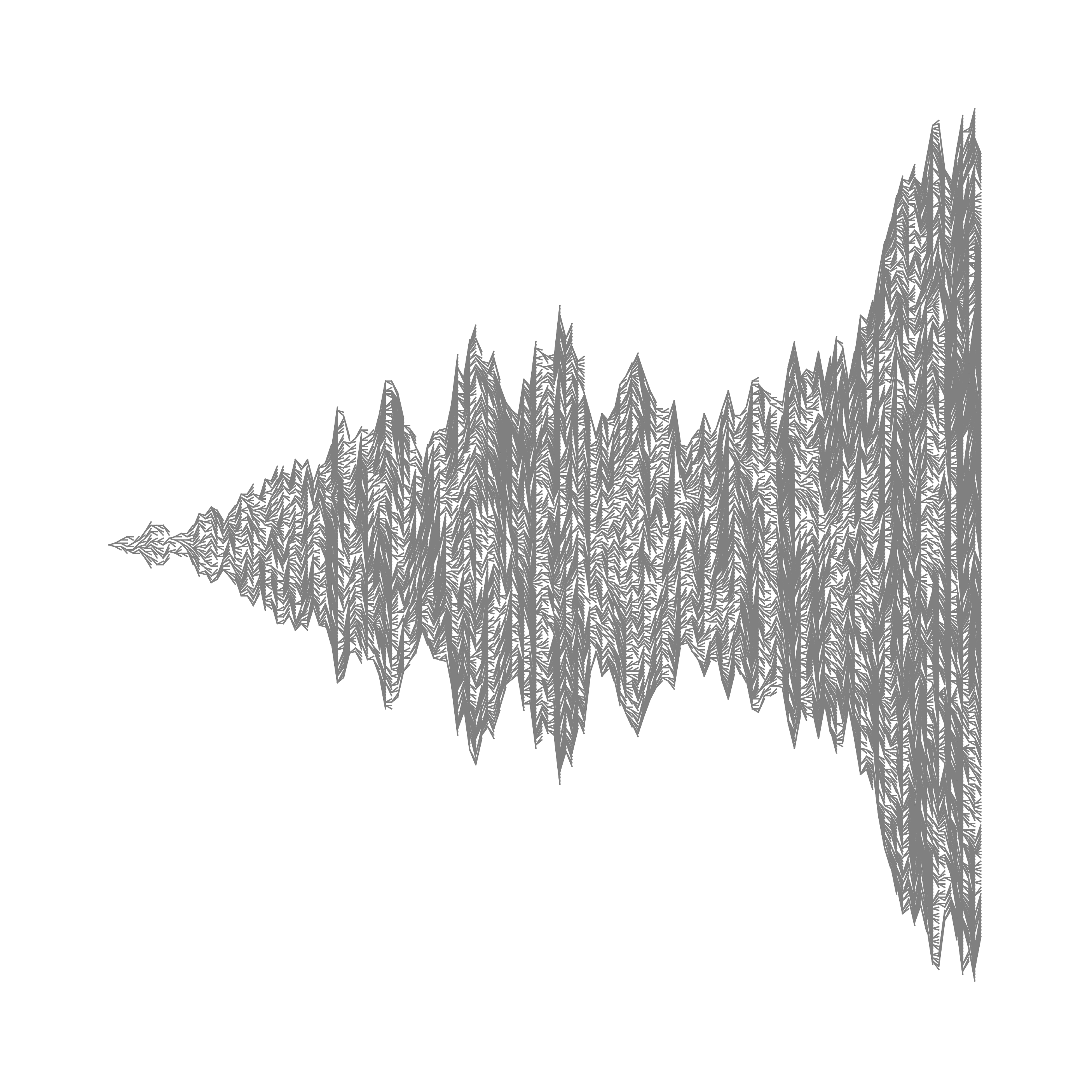}
   \includegraphics[width=5cm,angle=-1]{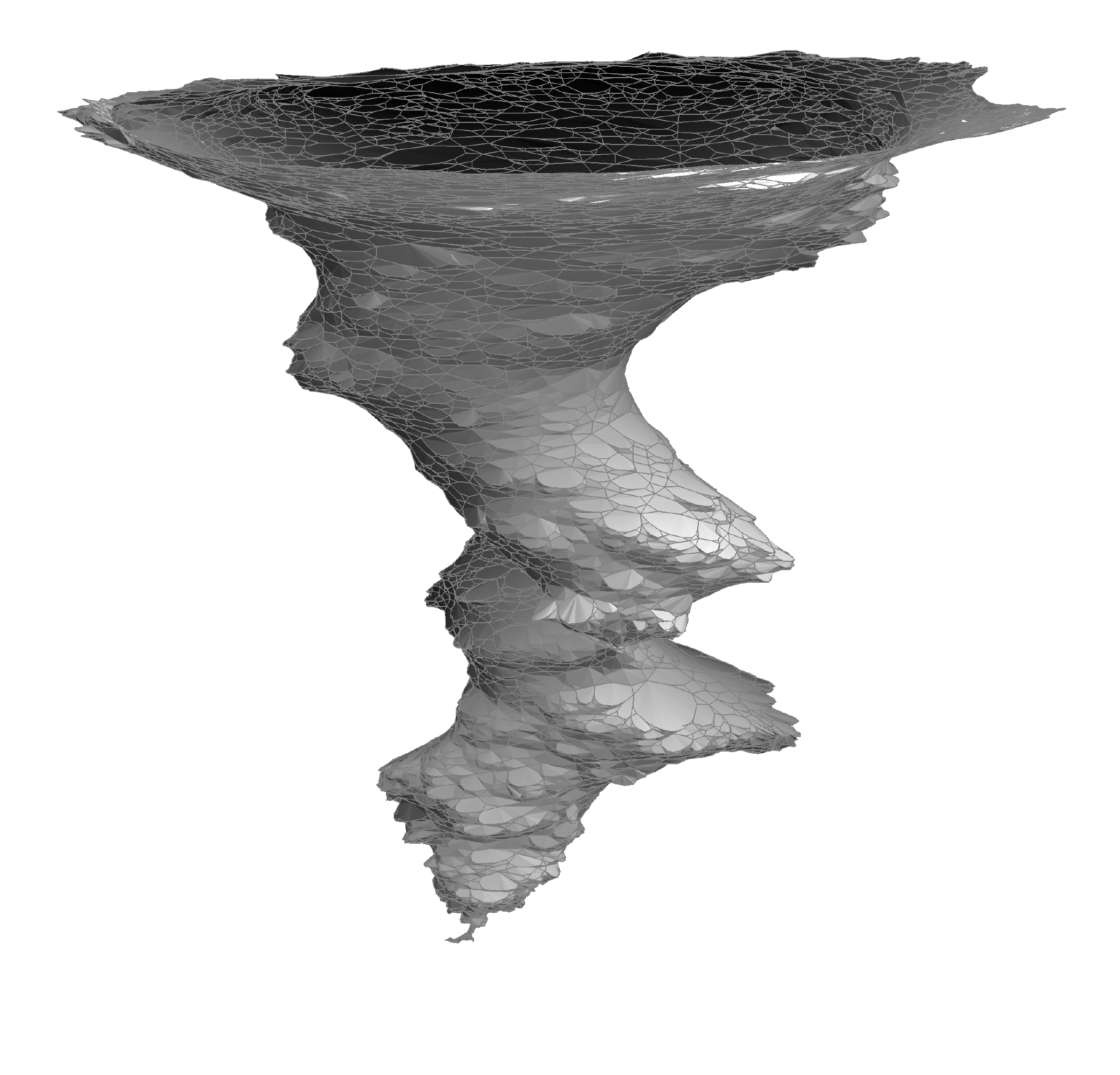}
 \caption{A piece of a large random Galton--Watson tree with finite variance represented in different ways (from left to right): a spring-electrical embedding, a layered representation and a 3D embedding of its associated causal map. \label{fig:nice}}
 \end{center}
 \end{figure}

\section{Introduction}
The \textbf{causal dynamical triangulation} (CDT) was 
introduced by theoretical physicists Jan Ambj\o rn and Renate Loll as a discrete model of Lorentzian quantum gravity in which space and time play different roles \cite{AmLo98}. Time is represented by a partition of the $d+1$-dimensional model into a sequence of $d$-dimensional layers with increasing distances from the origin. Although this model has been the subject of extensive numerical investigation \cite{AGJL12,AmJuLo05}, especially in dimension $2+1$ and $3+1$, very little is known analytically, let alone rigorously. 

In the case of dimension $1+1$, the $1$-dimensional layers are simply cycles, and causal triangulations are in bijection with plane trees, see e.g. \cite{DJW10,SYZ13}. Figure \ref{fig:causal-arbre} below illustrates the mechanism used to build a causal triangulation from a plane tree $\tau$: We first add the horizontal connections between successive vertices in each layer to obtain a planar map $ \mathsf{Causal}(\tau)$  living on the sphere, and then triangulate the  non-triangular faces of this map as shown in the drawing to obtain the triangulation $\mathsf{CauTrig}(\tau)$. See Section \ref{sec:causaltrig} and  \cite[Section 2.3]{DJW10} or \cite[Section 2.2]{SYZ13} for more details.

\vspace{1em}
 \begin{figure}[!h]
  \begin{center}
  \includegraphics[width=14cm]{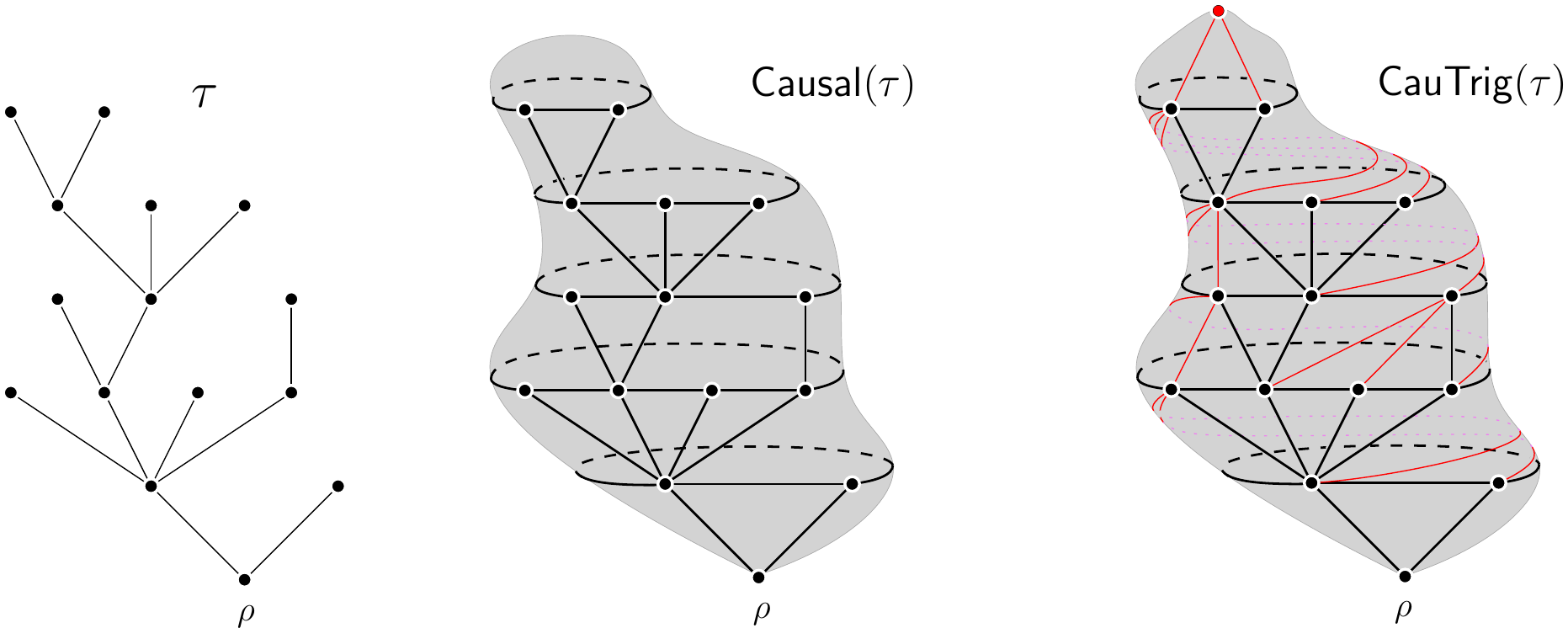}
  \caption{Left: A plane tree $\tau$. In any plane tree, there is a distinguished cyclic ordering of the vertices in each level. Center: The causal planar graph $ \mathsf{Causal}(\tau)$ built from $\tau$ by adding the horizontal connections between successive vertices in each level. Right: The causal triangulation $ \mathsf{CauTrig}(\tau)$ built from $ \mathsf{Causal}(\tau)$ by further triangulating the non-triangular faces from their top-right corners. \label{fig:causal-arbre}}
  \end{center}
  \end{figure}
\vspace{-1em}

The maps $\mathsf{Causal}(\tau)$ and $\mathsf{CauTrig}(\tau)$ are qualitatively very similar.
 We shall focus in this article on the model $ \mathsf{Causal}(\tau)$ (mainly to simplify our drawings) and refer the reader to Section \ref{sec:causaltrig} for extensions of our results to other models including causal triangulations. 

The geometry of large random plane trees is by now very well understood \cite{Ald91a,BK06,DLG02}. However, we shall see that causal maps have geometric and spectral properties that are dramatically different to the plane trees used to construct them. Indeed, the causal maps have much more in common with uniform random planar maps \cite{LeGallICM} such as the UIPT than they do with random trees.

\paragraph{Setup and results.} Suppose that $ \tau$ is a finite plane tree. We can associate with it a finite planar map (graph) denoted by $ \mathsf{Causal}(\tau)$ by adding the `horizontal' edges linking successive vertices in the cyclical ordering of each level of the tree as in Figure \ref{fig:causal-arbre}. If $\tau$ is an infinite, locally finite plane tree, performing the same operation yields an infinite planar map with one end, see Figure \ref{fig:nice}. 
%
The distance between a vertex $v$ of $\tau$ and the root $\rho$, called the \textbf{height} of $v$, is clearly equal in the two graphs $\tau$ and $\mathsf{Causal}(\tau)$. 
Thus, a natural first question is to understand how the distances between pairs of vertices at the \emph{same} height are affected by the addition of the horizontal edges in the causal graph.
We formalize this as follows: Let $\tau$ be a plane tree with root $\rho$. Let $[\tau] _{k}$ be the subtree spanned by the vertices of height at most $k$ and let $\partial [\tau]_{k}$ be the set of vertices of height exactly $k$. We define the \textbf{girth} at height $r$ of $\mathsf{Causal}(\tau)$ to be $$ \mathsf{Girth}_{r}\bigl(\mathsf{Causal}(\tau)\bigr) := \sup \left\{ \mathrm{d}_{\mathrm{gr}}^{ \mathsf{Causal}(\tau)}(x,y) : x,y \in \partial [\tau]_{r}\right\}, \quad \mbox{ where }\sup \varnothing =0,$$ and where $ \mathrm{d}_{\mathrm{gr}}^{ G}$ denotes the graph distance in the graph $G$. The triangle inequality yields the trivial bound $ \mathsf{Girth}_{r}(\mathsf{Causal}(\tau) ) \leq 2r$, so that the girth grows at most linearly.

 We will focus first on the case that the underlying tree $\tau$ is a random Galton-Watson tree whose offspring distribution $\mu$ is critical (i.e., has mean $1$) and has finite variance $0<\sigma^{2}<\infty$.
 The classical CDT model is related to the special case in which $\mu$ is a mean $1$ geometric distribution.
  Let $T$ be a $\mu$-Galton-Watson tree (which is almost surely finite) and let $T_{\infty}$ be a $\mu$-Galton-Watson tree conditioned to survive forever \cite{AD13,Kes86}. Let 
  $\Cinfty=\mathsf{Causal}(T_\infty)$. It is well-known that $ \# \partial [T_{\infty}]_{r} \approx r$ under the above hypotheses on $\mu$. 
 Our first main result states that the addition of the horizontal edges to the causal graph makes the girth at height $r$ smaller, but only by a subpolynomial factor. 

\begin{theorem}[Geometry of generic causal maps]  \label{thm:geometry} Let $\mu$ be critical and have finite non-zero variance. Then
$$  (i) \quad  \frac{\log \mathsf{Girth}_{r}( \Cinfty )}{\log r} \xrightarrow[r\to\infty]{(a.s.)} 1 \quad \mbox{ and } \quad (ii) \quad \frac{ \mathsf{Girth}_{r}( \Cinfty )}{ r} \xrightarrow[r\to\infty]{( \mathbb{P})} 0.$$
\end{theorem}

\hspace{0.4em}

A corollary of item $(ii)$ of Theorem \ref{thm:geometry} is that every geodesic between any two points at height $r$ in $\Cinfty$ stays within a strip of vertices at height $r\pm o(r)$ with high probability. This in turn implies that the scaling limit of $r^{-1} \cdot \Cinfty$  (in the local Gromov--Hausdorff sense)  is just a single semi-infinite line $( \mathbb{R}_{+}, |\cdot|)$. In other words, the metric in the horizontal (space) direction is collapsed relative to the metric in the vertical (time) direction, leading to a degenerate scaling limit.

The proof of item $(i)$ is based on a block-renormalisation argument and also yields the quantitative result that $\mathsf{Girth}_{r}( \Cinfty ) \geq r e^{-O(\sqrt{\log r})}$ as $r\to\infty$ almost surely (Proposition \ref{prop:quantgeometry}). On the other hand, item $(ii)$ uses the subadditive ergodic theorem and is not quantitative.


	 Once the geometry of $ \Cinfty$ is fairly well understood, we can apply this geometric understanding to study its spectral properties. We first show that $ \Cinfty$ is almost surely recurrent (Proposition \ref{prop:recurrence}) generalizing the result of \cite{DJW10}. Next, we apply Theorem \ref{thm:geometry} to prove the following results, the first of which completes the work of \cite{DJW10}. Given a connected graph $G$ and a vertex $x$, we denote by $\mathbf{P}_{G,x}$ the law of the simple random walk $(X_{n})_{n\geq 0}$ started at $x$ and denote by $P_G^{t}(x,x)$ the $t$-step return probability to $x$. The \textbf{spectral dimension} $d_{s}$ of a connected graph $G$ is defined to be 
\[
d_s(G)= \lim_{n\to\infty}-\frac{2\log P_G^{2n}(x,x)}{\log n}
\]
  should this limit exist (in which case it does not depend on $x$). We also define the \textbf{typical displacement exponent} $\nu=\nu(G)$ of the connected graph $G$ by
\[
\lim_{n\to\infty}\mathbf{P}_{G,x}\bigl(n^{\nu-\eps} \leq  \mathrm{d}_\mathrm{gr}^G(x,X_n)\leq n^{\nu+\eps}\bigr) =1 \text{ for every $\eps>0$}
\]
should such an exponent exist (in which case it is clearly unique and does not depend on $x$). We say that $G$ is \textbf{diffusive} for simple random walk if the typical displacement exponent $\nu(G)$ exists and equals $1/2$. 

\begin{theorem}[Spectral dimension and diffusivity of generic causal maps] \label{thm:spectral} Let $\mu$ be critical with finite non-zero variance. Then \[d_s(\Cinfty)=2 \quad \text{ and } \quad \nu(\Cinfty)=1/2\] almost surely. In particular, both exponents exist almost surely. 
\end{theorem}

Note that these exponents are \emph{not} the same as the underlying tree $T_\infty$, which has $d_s=4/3$ and $\nu=1/3$ \cite{BK06,DJW07,MR2407556,Kes86}.


The central step in the proof of Theorem \ref{thm:spectral} is to prove that the exponent $ \mathfrak{r}$ governing  the growth of the resistance between $\rho$ and the boundary of the ball of radius $n$ in $ \Cinfty$, defined by $ \mathsf{R_{eff}}(\rho \leftrightarrow \partial [T_\infty]_n;\, \Cinfty) = n^{ \mathfrak{r}+o(1)}$,  is $ \mathfrak{r}=0$.
In fact, we prove the following quantitative subpolynomial upper bound on the resistance growth.  This estimate is established using geometric controls on $ \Cinfty$ and the method of random paths \cite[Chapter 2.5]{LP10}. It had previously been open to prove any sublinear upper bound on the resistance.


\begin{theorem}[Resistance bound for generic causal maps ]
\label{thm:resistance}
Suppose $\mu$ is critical and has finite non-zero variance. Then there exists a constant $C$ such that almost surely for all $r$ sufficiently large we have 
 \[ \mathsf{{R}_{eff}}(\rho \leftrightarrow \partial [T_\infty]_r;\, \Cinfty) \leq e^{C \sqrt{\log r}}. \]

\end{theorem}

Theorem \ref{thm:spectral} can easily be deduced from Theorem \ref{thm:resistance} by abstract considerations. Indeed, by classical properties of Galton--Watson trees, the volume growth exponent $ \mathfrak{g}$, defined by  $ \#  \mathrm{Ball}(\rho,n) = n^{ \mathfrak{g}+o(1)}$, is easily seen to be equal to $2$. For recurrent graphs, the spectral dimension and typical displacement exponent can typically be computed from the volume growth and resistance growth exponents via the formulas
\[ d_{s} =  \frac{2 \mathfrak{g}}{ \mathfrak{g}+ \mathfrak{r}},\qquad \text{ and } \qquad d_s =2\nu  \mathfrak{g},\]
which yield $d_s=2$ and $\nu=1/\mathfrak{g}$ whenever $\mathfrak{r}=0$. 
Although this relationship between exponents holds rather generally (see \cite{BJKS08,K10,KM08}), things become substantially simpler in our case of $\mathfrak{r}=0,$ $\mathfrak{g}=2$ and we include a direct derivation. 
Indeed, in this case it suffices to use the inequalities
\[ d_s \geq 2-2\mathfrak{r}, \quad d_s \leq 2\nu \mathfrak{g}, \quad \text{ and } \quad \mathfrak{g}<\infty \Rightarrow \nu \leq 1/2,\]
which are more easily proven and require weaker controls on the graph. 
Let us note in particular that the \emph{upper bounds} on $d_s$ and $\nu$ are easy consequences of the Varopoulos-Carne bound 
 and do not require the full machinery of this paper.

\subsection{The $\alpha$-stable case.} Besides the finite variance case, we also study the case in which the offspring distribution $\mu$ is critical  and is ``$\alpha$-stable'' in the sense that it satisfies the asymptotic\footnote{Here $f(k) \sim g(k)$ means that $f(k)/g(k)\to1$ as $k\to+\infty$.}
  \begin{eqnarray} \mu([k, \infty)) \quad \underset{k \to \infty}{\sim}\quad  c\ k^{-\alpha}, \quad \mbox{ for some }  c >0 \mbox{ and }\alpha  \in (1,2).  \label{eq:defstable}\end{eqnarray}
In particular the law $\mu$ is in the strict domain of attraction of the totally asymmetric $\alpha$-stable distribution (we restrict here to polynomially decaying tails to avoid technical complications involving slowing varying functions). The study of such causal maps is motivated by their connection to uniform random planar triangulations. Indeed, Krikun's skeleton decomposition \cite{Kri04} identifies an object related to the stable causal map  with exponent $\alpha = \frac{3}{2}$ inside the UIPT, see Section \ref{sec:carpet}.

We still denote by $T_{\infty}$ the $\mu$-Galton--Watson tree conditioned to be infinite (the dependence in $\mu$, and hence in $\alpha$, is implicit), and denote by $\Cinfty$ the associated causal map. The geometry of $\mu$-Galton--Watson trees with critical ``$\alpha$-stable'' offspring distribution is known to be drastically different from the finite variance case. In particular, the size of the $n$th generation of $T_{\infty}$ is of order $n^{ \frac{1}{\alpha-1}}$ rather than $n$, and the the scaling limit is given by the (infinite) stable tree of Duquesne, Le Gall and Le Jan \cite{DLG02}, rather than the Brownian tree of Aldous \cite{Ald91a}.

We prove that there is a further pronounced difference occuring when one investigates the associated causal maps. Namely, while the girth at height $r$ was strictly sublinear in the finite variance case, it is linear in the $\alpha$-stable case. In particular, we have the following analog of Theorem \ref{thm:geometry}.

\begin{theorem}[Geometry of stable causal maps] \label{thm:geometrystable} If $\mu$ is critical and satisfies \eqref{eq:defstable} then we have
$$ (i) \quad   \frac{\log \mathsf{Girth}_{r}( \Cinfty)}{\log r} \xrightarrow[r\to\infty]{(a.s.)} 1 \quad \mbox{ and } \quad (ii) \quad   \lim_{ \varepsilon \to 0} \inf_{ r \geq 1} \mathbb{P}\left(\mathsf{Girth}_{r}(\Cinfty) \geq \varepsilon r\right)  = 1.$$
\end{theorem}

Similar to Theorem \ref{thm:geometry}, the proof of this theorem uses a block-renormalisation argument. 
 We  conjecture that in fact $ r^{-1}\mathsf{Girth}_{r}(\Cinfty)$ converges in distribution and more generally that $r^{-1} \cdot  \Cinfty$ converges in the local Gromov--Hausdorff sense. 
 These questions will be addressed in a forthcoming work of the first author. This theorem (and its proof) in fact have direct consequences in the theory of uniform random planar triangulations, using Krikun's skeleton decomposition; see Section \ref{sec:comments} for further details. 

 The adaptation of the techniques used to prove Theorem \ref{thm:spectral} here yields that 
 the resistance exponent satisfies 
\begin{equation}
\label{eq:stableresistancebound}
  \mathfrak{r} \leq  \frac{2- \alpha}{\alpha-1},
\end{equation}
while the volume growth exponent is known to be $ \mathfrak{g} = \frac{\alpha}{\alpha-1}$ \cite{CK08}.
   Notice that this bound is only useful in the range $\alpha \in (3/2,2)$ since we always have $ \mathfrak{r} \leq 1$. This witnesses that our understanding of the spectral properties of $ \Cinfty$ in the $\alpha$-stable case is much less advanced than in the finite variance case. 
The bound \eqref{eq:stableresistancebound} becomes much more interesting in the case of the $\alpha$-stable \emph{causal carpet}, which we expect to really have polynomial resistance growth; see Section \ref{sec:carpet} for further discussion. 
We remark that the spectral properties of the tree $T_\infty$ have been studied by Croydon and Kumagai \cite{CK08}, who prove in particular that $T_\infty$ has spectral dimension $2\alpha/(2\alpha-1)$ almost surely.

We are embarrassed to leave the following question open:
\begin{open} Suppose that $\mu$ is critical and satisfies \eqref{eq:defstable}. Is $ \Cinfty$ a.s.~transient?
\end{open}

\paragraph{Organization.} The paper is organized as follows. In Section \ref{sec:halfplane} we present the renormalisation technique that enables to bound from below the girth of causal graphs in a ``quarter-plane'' model carrying more independence than $\Cinfty$. This technique is rather general and we hope the presentation will make it easy to adapt to other settings. We also present the subadditive argument (Section \ref{sec:subadditive}) which gives the sublinear girth in the case of finite variance offspring distribution.
Section \ref{sec:width} is then devoted to the careful proof of Theorem \ref{thm:geometry} and \ref{thm:geometrystable}, which is done by dragging the quarter-plane estimates through to the original model $\Cinfty$. In Section \ref{sec:spectral}, we use the geometric knowledge gathered thusfar to prove Theorem \ref{thm:resistance} and deduce Theorem \ref{thm:spectral}. Section \ref{sec:comments} is devoted to extensions and comments.
\medskip

\paragraph{Acknowledgments:} NC acknowledges support from the Institut Universitaire de France, ANR Graal (ANR-14-CE25-0014), ANR Liouville (ANR-15-CE40-0013) and ERC GeoBrown. TH and AN were supported by ISF grant 1207/15 and ERC grant 676970 RandGeom. TH was also supported by a Microsoft Research PhD Fellowship and he thanks Tel Aviv University and Universit\'e Paris-Sud Orsay for their hospitality during visits in which this work was carried out. TH also thanks Jian Ding for bringing the problem of resistance growth in the CDT to his attention. Lastly, we warmly thank the anonymous referees for many valuable comments on the manuscript.
%

\vspace{0.3cm}

\begin{center} \hrulefill \fbox{\begin{minipage}{12cm}\textit{For the rest of the paper,  $\mu$ will be a fixed critical offspring distribution. Furthermore, we will always assume  either that $\mu$ has a finite, positive variance, or else that \eqref{eq:defstable} holds for some $\alpha \in (1,2)$. We refer to these two cases as the \emph{finite variance} and \emph{$\alpha$-stable} cases respectively. To unify notation, we let $\beta=1$ in the finite variance case and  $\beta = \frac{1}{\alpha-1}>1$ in the $\alpha$-stable case.} \end{minipage}}\hrulefill  \end{center}

\vspace{0.1cm}

\section{Estimates on the quarter-plane model}
\label{sec:halfplane}
The goal of this section is to study the girth of random causal graphs. For this, we first define a ``quarter-plane'' model carrying more symmetries and independence properties than $T_{\infty}$. We then define the notion of a \emph{block} and establish the key renormalisation lemma than enables us to lower bound the width of a block (Proposition \ref{prop:renorm}). The outcome of this renormalisation procedure is slightly different depending on whether $\mu$ has finite variance or is ``$\alpha$-stable''. These estimates will later be transferred to the actual model $ \Cinfty$ in Section \ref{sec:width}. In Section \ref{sec:subadditive} we present the subadditive argument for the quarter-plane model (Proposition \ref{prop:subadditive}) which will enable us to prove that the width is sublinear in the finite variance case. \bigskip

Before presenting the quarter-plane model, let us start by recalling a few standard estimates on critical Galton--Watson trees. Recall that $\mu$ is always a \emph{critical} offspring distribution and recall the definition of $\beta$ above. The famous estimate of Kolmogorov and its extension to the stable case by Slack \cite{Slack68} states that
 \begin{eqnarray} \label{eq:kolmogorov} \mathbb{P}( \mathsf{Height}(T) \geq n) \sim c\, n^{-\beta}, \quad \mbox{for some }c>0, \quad \mbox{ as } n \to \infty.  \end{eqnarray}
Furthermore, conditionally on non-extinction at generation $n$, the total size of generation $n$ converges after rescaling by $n^{\beta}$ towards a non-zero random variable (Yaglom's limit and its extension by Slack \cite{Slack68}):
 \begin{eqnarray} \label{eq:yaglom}  \left(n^{-\beta}\# \partial [T]_{n} \in \cdot  \mid \mathsf{Height}(T) \geq n\right) \xrightarrow[n\to\infty]{(d)} \mathcal{X},  \end{eqnarray} where $ \mathcal{X}$ is an explicit positive random variable (but whose exact distribution will not be used in the sequel).
\subsection{The block-renormalisation scheme}\label{sec:blocks}

\paragraph{The quarter-plane model.} We consider a sequence $T_{1},T_{2}, \dots$ of independent and identically distributed $\mu$-Galton--Watson trees. We can then index the vertices of this forest by $\{1,2,\ldots\}\times \{0,1,\ldots\}$ in an obvious way as depicted in the Figure \ref{fig:block} below.
\vspace{0.5em}
\begin{figure}[!h]
 \begin{center}
 \includegraphics[width=14cm]{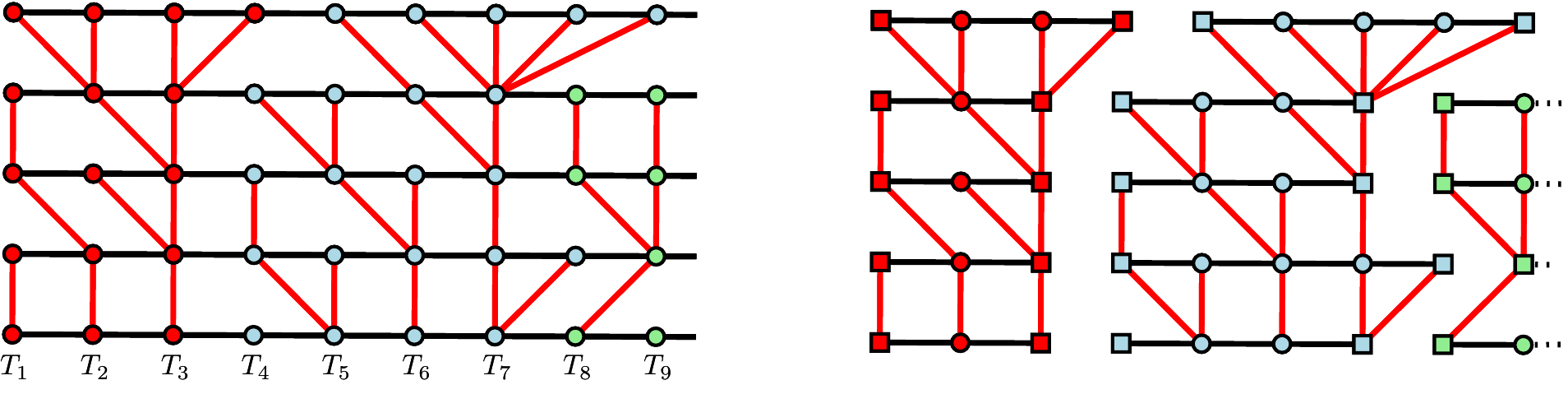}
 \caption{ \label{fig:block} The layer of height $4$ in the random graph obtained from an iid sequence of $\mu$-Galton--Watson trees. This layer can be decomposed into blocks of height $4$ and we represented the first three blocks in this sequence. The vertices on the left and right sides of the blocks are denoted with squares.}
 \vspace{-1em}
 \end{center}
 \end{figure}

Adding the horizontals edges $(i,j) \leftrightarrow (i+1,j)$ forms an infinite planar map (graph) which we call the \textbf{quarter-plane model} 
and denote by  $\mathcal{Q}_\infty$.
  Let $\xi_r\geq 1$ be minimal such that $T_{\xi_{r}}$ reaches height $r$. The \textbf{block of height $r$}, denoted $\mathcal{G}_r$, is defined to be the subgraph of $\mathcal{Q}_\infty$ induced by
 the vertices of $T_{1}, \dots , T_{\xi_{r}}$ at height less than or equal to $r$. That is, $\mathcal{G}_r$ consists of all the vertices of  $T_{1}, \dots , T_{\xi_{r}}$ at height less than or equal to $r$ and all the edges of $\mathcal{Q}_\infty$ (both horizontal and vertical) between them.
 See Fig.~\ref{fig:block} for an illustration. 
  Clearly, we can speak of the two sets of vertices belonging respectively to the left side and right side of the block $ \mathcal{G}_{r}$, that is, the set of $r+1$ left most vertices and the set of $r+1$ right most vertices of the block. We define the \textbf{width} of $ \mathcal{G}_{r}$, denoted by $\mathsf{Width}( \mathcal{G}_{r})$, to be the minimal graph distance (in $ \mathcal{G}_{r}$) between a vertex in the left side and a vertex in the right side of $ \mathcal{G}_{r}$, see Fig.~\ref{fig:block}.  The width of a block is not uniformly large when $r$ is large: Indeed,  the first tree $T_{1}$ may actually reach the level $r$ in which case $ \mathsf{Width}( \mathcal{G}_{r}) = 0$. However, we will see that a large block typically has a large width. To this end, we consider the \emph{median} of $\mathsf{Width}( \mathcal{G}_{r})$:

\begin{definition} For each $r \geq 1$ let $f(r)$ be the median width of $\mathcal{G}_r$, that is, the largest number such that 
\[ \mathbb{P}\big( \mathsf{Width}( \mathcal{G}_{r}) \geq f(r)\big) \geq 1/2.\]
\end{definition}

As usual, the dependence on the offspring distribution is implicit in the notation. Obviously the value $1/2$ is not special. Note that, depending on $\mu$, one might have that $f(r)=0$ for small values of $r$. 
On the other hand, $\mathsf{Width}(\mathcal{G_r})$ is bounded deterministically by $2r$ since all vertices in the top layer of $\mathcal{G}_r$ share a common ancestor in level zero, so that $f(r)\leq 2r$ also. 
Our main technical result shows that $f(r)$ is always roughly linear, more precisely:

\begin{theorem}[$f(r)$ is almost linear] \label{thm:main} If $\mu$ is critical and satisfies \eqref{eq:defstable} then there exists $c >0$ such that  
\begin{equation}
\label{eq:flinear}
f(r) \geq c \,r\end{equation}
for all $r$ sufficiently large.
On the other hand, if $\mu$ is critical and has finite non-zero variance then there exists $C>0$ such that 
\begin{equation}\label{eq:falmostlinear}
f(r) \geq r \exp( - C \sqrt{ \log r})\end{equation}
for all $r$ sufficiently large.
\end{theorem}

The above theorem is an analytic consequence of the following proposition which encapsulates the renormalisation scheme. Recall the definition of $\beta \geq 1$ at the end of the Introduction.

\begin{proposition} \label{prop:renorm} There exists $c >0$ such that for any $1 \leq m \leq c \cdot  r$ we have
\begin{equation}
\label{eq:renorm}
 f(r)  \geq c \cdot  \min \Big\{ m ;\,   (r/m)^{\beta} f(m) \Big\}.
 \end{equation}
\end{proposition}

The proof of this proposition relies on a renormalisation scheme in which $ \mathcal{G}_{r}$ is split into smaller blocks distributed as $ \mathcal{G}_{m}$ for $0 \leq m \leq r$. Before starting the proof, let us introduce some notation. Let $m,h \geq 0$, and consider the layer of thickness $m$  between heights $h$ and $h+m$ in the quarter-plane model $\mathcal{Q}_\infty$. This layer is composed of a sequence of blocks of height $m$ which we denote by $ \mathcal{G}_{m}(i,h)$ for $i \geq 1$.  For any fixed $m,h \geq 0$, these blocks are of course independent  and distributed as $ \mathcal{G}_{m}$ (indeed, $\mathcal{G}_m(1,0)$ is equal to $\mathcal{G}_m$). When $h+m \leq r$, we denote by $N_{r}(m,h)$ the maximal $i$ such that the block $ \mathcal{G}_{m}(i,h)$ is a subblock of $ \mathcal{G}_{r}$.

\vspace{1.5em}
\begin{figure}[!h]
 \begin{center}
 \includegraphics[width=16cm]{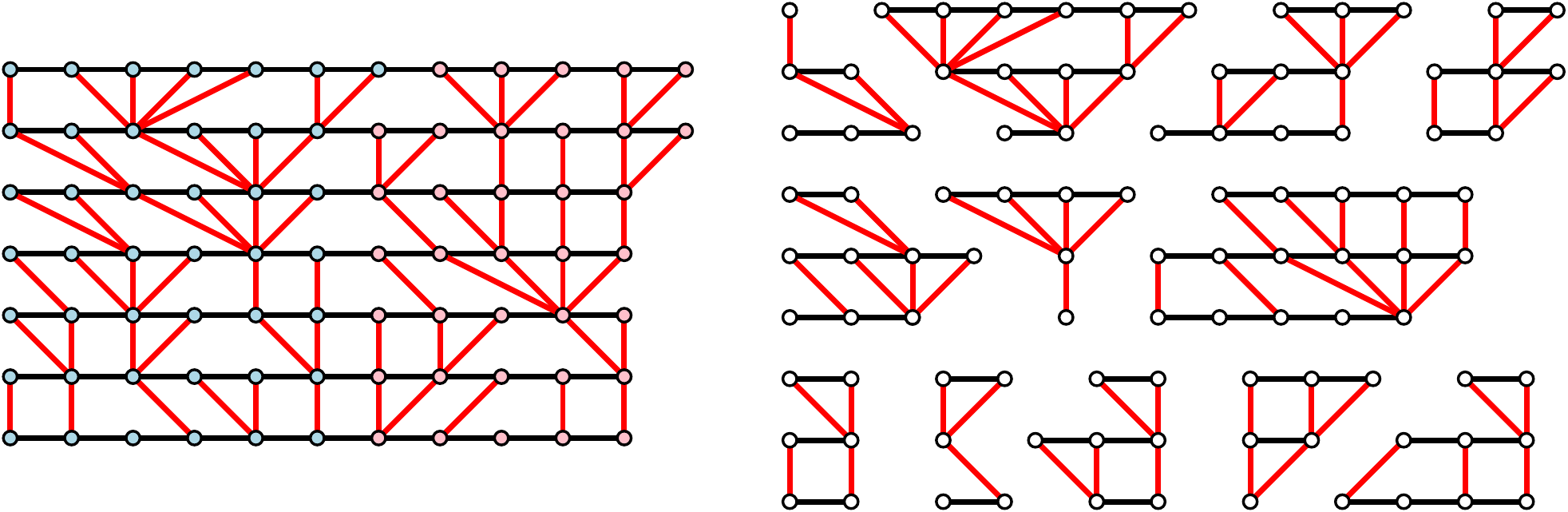}
 \caption{ \label{fig:renormalisation} Decomposing a big block into smaller blocks: On the left we see the blocks $ \mathcal{G}_{6}(1,0)$ and $ \mathcal{G}_{6}(2,0)$. These two blocks are further decomposed into the blocks $ \mathcal{G}_{2}(i,h)$ on the right figure for $h \in \{0,2,4\}$ and $i \geq 1$.}
 \end{center}
 \end{figure}
\vspace{-1.5em}

We also recall the following classical Chernoff-type bound for sums of indicator random variables \cite[Corollary A.1.14]{AlonSpencer}, which will be used throughout the paper: 
For every $\eps>0$ there exists a constant $c_\eps>0$ such that for every $k\geq 1$ and every sequence  $X_1,\ldots X_k$ of mutually independent $\{0,1\}$-valued random variables, the sum $Y= \sum_{i=1}^k X_i$ satisfies the bound
\begin{equation}
\label{eq:ASChernoff}
\P\Bigl(\bigl|Y-\E[Y]\bigr| \geq \eps\, \E[Y] \Bigr) \leq 2 e^{-c_\eps \E[Y]}.
\end{equation}

We are now ready to proceed with the proof of Proposition \ref{prop:renorm}.

\proof[Proof of Proposition \ref{prop:renorm}] 
We will prove that 
there exists $c >0$ such that 
\begin{equation}
\label{eq:renorm}
 f(r)  \geq c \cdot  \min \Big\{ m ;\,   (r/m)^{\beta} f(2m) \Big\}.
 \end{equation}
 for every $1 \leq m \leq c \cdot  r$ that divides $r$. The full proof of the claim as originally stated (in which $m$ replaces $2m$ and $m$ is not assumed to divide $r$) is very similar but requires messier notation.

   We  bound from below the width of the block $ \mathcal{G}_{r}$ using the widths of the blocks $ \mathcal{G}_{2m}(i,h)$ for $h$ of the form $\ell \cdot m$ with $0 \leq \ell \leq (r/m)-2$. Suppose that we pick a point $x$ on the left side of $ \mathcal{G}_{r}$, say at height $0 \leq j \leq r$. If $m/3 \leq j \leq r-m/3$ then we can clearly take $ 0 \leq \ell \leq  (r/m)-2$ such that $| \ell m -j| \geq m/3$ and $|(\ell+2)m-j| \geq m/3$. Otherwise, we either have that $ 0\leq  j < m/3$ and take $\ell=0$ or we have that $ r-m/3 < j \leq r$ and take $\ell =(r/m)-2$. We then have an alternative: either the shortest path from $x$ to the other side of $ \mathcal{G}_{r}$ stays in the layer between heights $\ell m$ and $(\ell +2)m$, or else it leaves it at some point. In the second case we know that the length of the path is at least $m/3$ by our assumption on $j$ and $\ell$ and since the graph distance between any two points in the graph is at least their height difference. On the other hand, in the first case, the length of such a path  is at least
$$ \sum_{i=1}^{N_{r}(2m,\ell m)} \mathsf{Width}\big( \mathcal{G}_{2m}(i, \ell m)\big).$$ This is because the path must cross, from left to right, every subblock of height $2m$ that is in that layer and that belongs to the block $ \mathcal{G}_{r}$. See Fig.~\ref{fig:trans}. We conclude that
\begin{equation}\label{eq:logic} \mathsf{Width}( \mathcal{G}_{r}) \geq \min \Big \{ {m \over 3} , \min _{0 \leq \ell \leq (r/m)-2} \sum_{i=1}^{N_{r}(2m,\ell m)} \mathsf{Width}\big( \mathcal{G}_{2m}(i, \ell m)\big) \Big \} \, .\end{equation}

\begin{figure}[!h]
 \begin{center}
 \includegraphics[width=10cm]{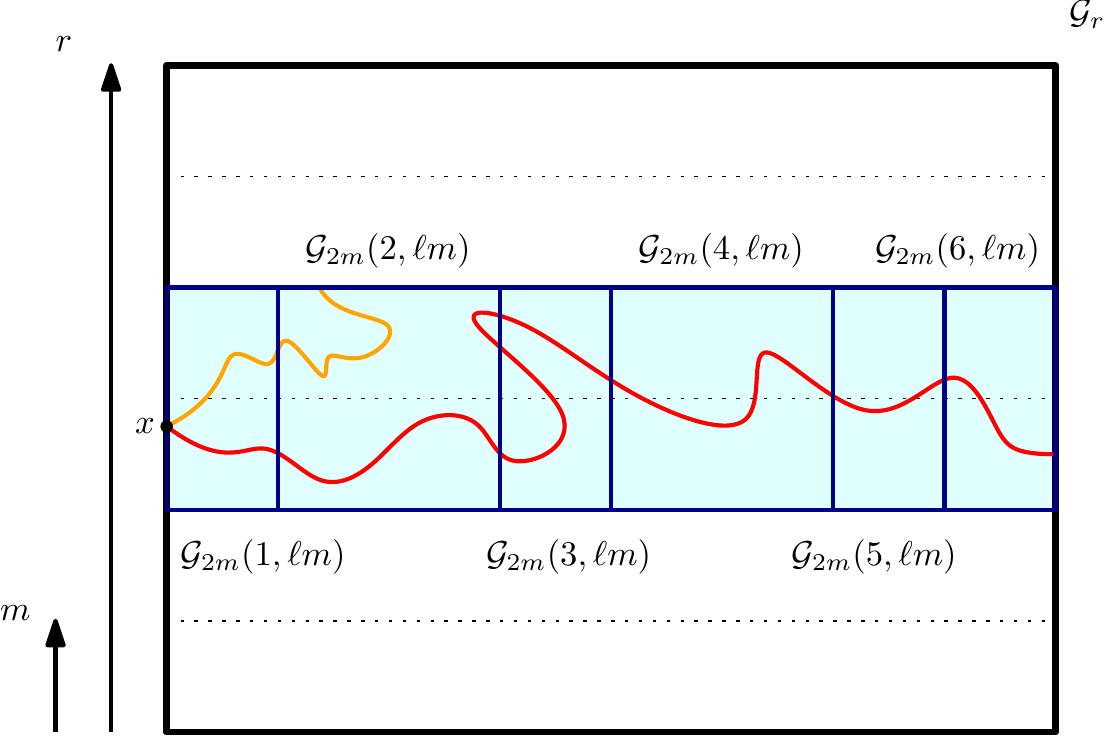}
 \caption{ 
 \small{\label{fig:trans} Illustration of the proof. Any point on the left-hand side of $ \mathcal{G}_{r}$ is well inside a layer of height $2m$ starting at some height $\ell m$. The geodesic from $x$ to the other side of the block either leaves this layer (in orange) or must traverse all the $N_{r}(2m, \ell m)$ sub-blocks (in red).}}
 \end{center}
 \end{figure}

 For fixed $h,m \geq 1$ the blocks $ \mathcal{G}_{2m}(i,h)$ are independent and distributed as $ \mathcal{G}_{2m}$. Thus, by the definition of the function $f(2m)$ and \eqref{eq:ASChernoff}, we have that 
$$ \mathbb{P}\left( \sum_{i=1}^{k} \mathsf{Width}( \mathcal{G}_{2m}(i,h)) \leq \frac{k \cdot f(2m)}{4}\right) \leq \mathbb{P}\Big( \mathrm{Binomial}(k,1/2) \leq k/4\Big) \leq e^{- \eta k}$$ for every $k\geq 1$, where $\eta >0$ is a constant independent of $k,h$ and $m$. Summing-up over all possibilities for $h = \ell \cdot m$ with $\ell \in \{0,\ldots,(r/m)-2\}$ we deduce that with probability at least $ 1- \frac{r}{m} e^{- \eta k}$ we have
 \begin{eqnarray} \label{eq:LD} \forall 0 \leq \ell \leq (r/m)-2, \quad \quad  \sum_{i=1}^{k} \mathsf{Width}( \mathcal{G}_{2m}(i, \ell m)) \geq \frac{k \cdot f(2m)}{4}. \end{eqnarray}
We now estimate $N_{r}(2m ,\ell m)$:
 \begin{lemma}  \label{lem:manymany} There exists $c>0$ such that for every $1 \leq m \leq c r$ we have
$$  \mathbb{P}\left(\min_{0 \leq \ell \leq (r/m)-2} N_{r}(2m, \ell m) \geq c \left( \frac r m\right)^{\beta}\right) \geq \frac{7}{8}.$$
\end{lemma}
\proof We first consider the analogous estimate in the case $m=0$. In this case, $N_{r}(0,h)$ is the just the number of vertices at height $h$ in the block $ \mathcal{G}_{r}$. We claim that we can find $c>0$ sufficiently small such that 
 \begin{eqnarray} \label{eq:CSBP}
 \mathbb{P}\left(\min_{0 \leq h \leq r} N_{r}(0,h) \geq c \cdot r^\beta \right) \geq 15/16 \end{eqnarray}
for every $r\geq 1$.  This kind of result is part of the folklore in the theory of branching processes (see e.g. \cite{DLG02}) but since we were not able to locate a precise reference for it we include a direct derivation at the end of this subsection (Lemma \ref{lem:CSBP}).

We now apply \eqref{eq:CSBP} to prove the claim in the statement of the lemma. Let $c$ be the constant from \eqref{eq:CSBP}.
Fix $0 \leq h \leq r$ and denote by $X(h,2m,k)$ the number of trees   whose origin in the line of height $h$ has index less than $k \geq 1$ and which reach height $2m$ relative to its starting height of $h$. With this notation we have $N_{r}(2m,h) = X(h,2m,N_{r}(0,h))$. For fixed $k \geq 1$ the random variable $X(h,2m,k)$ has binomial distribution with $k$ trials and success parameter $\mathbb{P}(\mathrm{Height}(T) \geq 2m)$. On the event $\{N_{r}(0,h) \geq c r^{\beta}\} \cap \{N_{r}(2m,h) \leq c' (r/m)^{\beta}\}$ we have $X(h,2m,\lceil c r^{\beta} \rceil) \leq c' (r/m)^{\beta}$. Thus, applying \eqref{eq:ASChernoff} and \eqref{eq:kolmogorov} we deduce that there exist constants $c'>0$ and $\delta>0$ such that
 \[ \P\Bigl(N_{r}(2m,h)\leq c'(r/m)^{\beta} \mbox{ and } N_{r}(0,h) \geq c \cdot r^\beta \Bigr) \leq \P\Bigl(X(h,2m,\lceil c r^{\beta} \rceil) \leq c' (r/m)^{\beta}\Bigr) \leq 
  e^{-\delta (r/m)^{\beta}}\] for every $r,m,$ and $h$. For sufficiently large values of $r/m$ we have that $ (r/m) e^{-\delta (r/m)^{\beta}} \leq 1/16$,  and can proceed to apply a union bound over values of $h$ of the form $\ell m$ for $\ell \in \{0, \ldots, (r/m)-2\}$. Indeed, gathering up the pieces above we have that
 \begin{align*}  \mathbb{P}( \exists 0 \leq \ell \leq (r/&m)-2 : N_{r}(2m, \ell m) \leq  c'(r/m)^{\beta})\\  &\leq \mathbb{P}\left( \min_{0 \leq h \leq r} N_{r}(0,h) \leq c \cdot r^{\beta}\right) +  \sum_{\ell=0}^{(r/m)-2} \mathbb{P}\left( N_{r}(2m,\ell m) \leq c'(r/m)^{\beta}, \,\, N_{r}(0,\ell m) \geq cr^{\beta}\right)\\ & \leq   \frac{1}{16} +  (r/m) e^{-\delta (r/m)^{\beta}} \leq \frac{1}{16}+ \frac{1}{16} = \frac{1}{8},  \end{align*} and this proves the lemma. \endproof

We now return to the proof of Proposition \ref{prop:renorm}. Let $c$ be the constant from Lemma~\ref{lem:manymany}.   We take $k = \lfloor c(r/m)^{\beta} \rfloor$ in \eqref{eq:LD} and assume that $r/m$ is large enough to ensure that $ (r/m) e^{- \eta k} \leq 1/8$. Using Lemma \ref{lem:manymany} and intersecting with the event in \eqref{eq:LD} we deduce by \eqref{eq:logic} that
 \begin{eqnarray*} \mathbb{P}\left( \mathsf{Width}( \mathcal{G}_{r}) <  \frac{m}{3} \wedge \frac{k\cdot f(2m)}{4}\right) &\leq&  \mathbb{P}\left(\min_{0 \leq \ell \leq (r/m)-2} N_{r}(2m, \ell m) < c \left( \frac r m\right)^{\beta}\right) + (r/m) e^{-\eta k} \\ & \leq & \frac{1}{8} + \frac{1}{8} = \frac{1}{4}.  \end{eqnarray*}
By definition of $f(r)$ we thus deduce that
 $$f(r) \geq \min \Big\{ \frac{m}{3} ;\,  \frac{c}{4} (r/m)^{\beta}f(2m)\Big\} \geq c'' \inf \big\{ m ;\, (r/m)^{\beta} f(2m)\big\},$$ for some $c''>0$ and every $r \geq m \geq 1$ such that $m$ divides $r$ and $r/m$ is sufficiently large.
  \endproof


\begin{proof}[Proof of Theorem \ref{thm:main} from Proposition \ref{prop:renorm}] 
Assume that $f$ satisfies the conclusions of Proposition \ref{prop:renorm} with the appropriate $\beta$ and that $f(r)>0$ for all  sufficiently large $r$, which is easily seen to be satisfied by our function $f$. (One way to prove this formally is to note that the width is zero if and only if $\min_{0\leq h \leq r} N_r(0,h) =1$, and then  apply Lemma \ref{lem:CSBP}, below. Easier and more direct proofs are also possible.)

 First suppose that $\beta >1$ (i.e.~that we are in the stable case) and let $k$ be an integer that is larger than both $c^{-1}$ and $c^{-1/(\beta-1)}$. Let $a_n = f(k^n)/k^n$. We claim that
\[
\liminf_{n\to\infty} a_n  > 0.
\]
Indeed, applying Proposition~\ref{prop:renorm} with $r=k^{n+1}$ and $m=k^{n}$ yields that
\[a_{n+1} \geq \min\left\{\frac{c k^n}{k^{n+1}},\, c k^{-1} \left(\frac{k^{n+1}}{k^n}\right)^\beta a_{n} \right\} 
\geq \min\left\{ \frac{c}{k},\, a_n\right\}, \]
and since $a_n>0$ for some $n\geq 1$ it follows by induction that $\liminf_{n\to\infty} a_n>0$ as claimed. This establishes the inequality \eqref{eq:flinear} for values of $r$ of the form $r=k^n$. The inequality \eqref{eq:flinear} follows for general values of $r$ by taking $m= k^{\lfloor \log_k r -1\rfloor}$ in Proposition~\ref{prop:renorm}. 

Now suppose that $\beta=1$, and let $k \geq 1/c$ be an integer.
We put $b_n =f(k^{n^2})/k^{(n-1)^2}$ and will show that
\[
\liminf_{n\to\infty} b_n  > 0.
\]
Indeed, applying Proposition~\ref{prop:renorm} with $r=k^{(n+1)^2}$ and $m=k^{n^2} \leq c k^{(n+1)^2}$ yields that
\[
b_{n+1} \geq \min \left\{ c\frac{k^{n^2}}{k^{n^2}},\,  c \frac{k^{(n+1)^2+(n-1)^2}}{k^{2n^2}} b_n \right\} = \min\{c,\, ck^2 b_n \} \geq \min\{c,\, b_n\},
\]
and since $b_n>0$ for some $n\geq 1$ it follows by induction that $\liminf_{n\to\infty} b_n>0$ as desired. This establishes the inequality \eqref{eq:falmostlinear} for values of $r$ of the form $r=k^{n^2}$. 
The inequality \eqref{eq:falmostlinear} for other values of $r$ follows by applying Proposition \ref{prop:renorm} with 
$m=k^{\left\lfloor \sqrt{\log_k r} -1 \right\rfloor^2}$. \qedhere
\end{proof}

\begin{remark}
With a little further analysis, it can be shown that (disregarding constants) the bound $f(r)\geq r e^{-O(\sqrt{\log r})}$ is the best that can be obtained from Proposition \ref{prop:renorm} in the case $\beta=1$. 
\end{remark}

We now owe the reader the proof of \eqref{eq:CSBP}:
\begin{lemma}
\label{lem:CSBP}
 With the notation of Lemma \ref{lem:manymany}, for any $ \varepsilon>0$ we can find $\delta >0$ such that for every $r \geq 1$ we have that
\[ \mathbb{P}\Bigl( \min_{0 \leq h \leq r} N_{r}(0,h) \leq \delta r^{\beta}\Bigr) \leq \varepsilon.\]
\end{lemma}
\proof Fix $ \varepsilon>0$. Recall that $T_{1}, T_{2}, \dots$ are independent $\mu$-Galton--Watson trees and that $T_{\xi_{r}}$ is the first of these trees reaching height $r$. We denote by $ \mathbf{X}_{i}(h)$ the total number of vertices at height $h$ belonging to the trees $T_{1}, \dots , T_{i}$ so that $X_{\xi_{r}}(h) = N_r(0,h)$. It suffices to show that if $\delta >0$ is sufficiently small then
\[\mathbb{P}\Bigl( \min_{0 \leq h \leq r}  \mathbf{X}_{\xi_{r}}(h) \leq \delta r^{\beta}\Bigr) \leq \varepsilon\] for all $r \geq 1$.
We start with two remarks. First, observe that the law of $\xi_{r}$ is a geometric random variable with success probability $ \mathbb{P}( \mathrm{Height}(T) \geq r)$. By \eqref{eq:kolmogorov} this success probability is asymptotic to $cr^{-\beta}$ as $r\to\infty$, and it follows that if  $ \eta >0$  is sufficiently small then $ \mathbb{P}(\xi_{r} \notin [\eta r^{\beta},  \eta^{-1} r^{\beta}]) \leq \varepsilon/3$ for all $r\geq 1$. Secondly, using  \eqref{eq:kolmogorov} again, it is easy to see that we can find $ \varepsilon' >0$ such that the height of $T_{\xi_{r}}$ is at least $(1+ \varepsilon')r$ with probability at least $ 1-\varepsilon/3$. Thus, by the union bound, it suffices to prove that if $\delta>0$ is sufficiently small then 
\begin{equation}
\label{eq:lem2desired}
\mathbb{P}\left( \left\{\min_{0 \leq h \leq r}  \mathbf{X}_{\xi_{r}}(h) \leq \delta r^{\beta}  \right\} \cap \left\{ \eta r^{\beta}<\xi_{r}< \eta^{-1} r^{\beta}\right\} \cap \left\{  \xi_{r} = \xi_{(1+ \varepsilon')r} \right \}\right) \leq \varepsilon/3
\end{equation} for every $r \geq 1$. 

Let $k,r\geq 1$. Let $\mathcal{W}_{r,k}$ be the event that $\xi_{r}=\xi_{(1+ \varepsilon')r} = k$, and let $\mathcal{U}_{r,k}$ be the event that exactly one of the trees $T_1,\ldots,T_k$ reaches height $(1+\eps')r$ while every other tree indexed by $\{1,\ldots, k\}$ reaches height strictly less than $r$. If $\sigma$ is a uniform random permutation of $\{1,\ldots, k\}$ independent of $T_1,T_2,\ldots$, notice the following equality of conditional distributions
\begin{equation}
\label{eq:lemma2distributionalequality}
\Bigl( \bigl(T_1,\ldots,T_k\bigr)  \mid \mathcal{U}_{r,k} \Bigr) \overset{d}{=} \Bigl( \bigl(T_{\sigma(1)},\ldots,T_{\sigma(k)}\bigr)  \mid \mathcal{W}_{r,k} \Bigr).
\end{equation}
%
 On the other hand, if we define $\mathcal{V}_{r,k}$ to be the event that \emph{at least one} of the $k$ trees $T_{1}, \dots , T_{k}$ reaches height $(1 + \varepsilon')r$, then a little calculation using \eqref{eq:kolmogorov} shows that  there exist constants $0<c_{1}<c_{2}<1$ (depending on $ \varepsilon'$ and $\eta$) such that
 \begin{eqnarray} \label{eq:condi}0<c_{1}<\mathbb{P}( \mathcal{U}_{r,k}) \leq \mathbb{P}( \mathcal{V}_{r,k})<c_{2}<1  \end{eqnarray}
for every $r \geq 1$ and all $\eta r^\beta \leq k \leq \eta^{-1} r^\beta$. 
We deduce that there exists a constant $c_3>0$ such that
 \begin{align*}   &\mathbb{P}\biggl( \biggl\{\min_{0 \leq h \leq r}  \mathbf{X}_{\xi_{r}}(h) \leq \delta r^{\beta}  \biggr\} \cap \biggl\{ \eta r^{\beta}<\xi_{r}< \eta^{-1} r^{\beta}\biggr\} \cap \left\{  \xi_{r} = \xi_{(1+ \varepsilon')r} \right \}\biggr) \\
 &\hspace{1.6cm}
  =  \sum_{ \eta r^{\beta} < k < \eta^{-1} r^{\beta}} \mathbb{P}\left( \left\{\min_{0 \leq h \leq r}  \mathbf{X}_{k}(h) \leq \delta r^{\beta}  \right\} \cap  \mathcal{W}_{r,k}\right)\\
 &\hspace{1.6cm}\hspace{1.6cm}
\underset{ \eqref{eq:kolmogorov}}{\leq} c_3
  \sup_{\eta r^{\beta} < k < \eta^{-1} r^{\beta}} \mathbb{P}\left(\min_{0 \leq h \leq r}  \mathbf{X}_{k}(h) \leq \delta r^{\beta}   \mid  \mathcal{W}_{r,k}\right)  \\
 &\hspace{1.6cm}\hspace{1.6cm}\hspace{1.6cm}
  \underset{\eqref{eq:lemma2distributionalequality}}{=}  c_3\sup_{\eta r^{\beta} < k < \eta^{-1} r^{\beta}} \mathbb{P}\left(\min_{0 \leq h \leq r}  \mathbf{X}_{k}(h) \leq \delta r^{\beta}   \mid  \mathcal{U}_{r,k} \right) \\
 &\hspace{1.6cm}\hspace{1.6cm}\hspace{1.6cm}\hspace{1.6cm}
  \underset{ \eqref{eq:condi}}{\leq}    \frac{c_3}{c_{1}}\sup_{\eta r^{\beta} < k < \eta^{-1} r^{\beta}} \mathbb{P}\left(\left\{\min_{0 \leq h \leq r}  \mathbf{X}_{k}(h) \leq \delta r^{\beta}  \right\}  \cap  \mathcal{V}_{r,k} \right)
 \end{align*}
 for every $r\geq 1$. 
But now one can easily estimate $\mathbb{P}\left(\min_{0 \leq h \leq r}  \mathbf{X}_{k}(h) \leq \delta r^{\beta}   \cap  \mathcal{V}_{r,k} \right)$: If $0 \leq h_{0} \leq r$ is the first height at which we have $ \mathbf{X}_{k}(h_{0}) \leq \delta r^{\beta}$, then by the Markov property of the branching process, the probability that one of the descendants  of the $  \mathbf{X}_{k}(h_{0})$ points at generation $h_{0}$ reach height $(1+ \varepsilon')r$ is bounded from above by $ \delta r^{\beta} \mathbb{P}( \mathrm{ht}(T) \geq \varepsilon' r)$. Using \eqref{eq:kolmogorov} again, we can choose $\delta>0$ small enough so that this probability is less than $  c_{1} \cdot \varepsilon /3$ for all $r \geq 1$. For this choice of $\delta$ we indeed have \eqref{eq:lem2desired}. 
\qed

\subsection{The dual width}
\label{sec:dualdiam}

In order to analyze resistances, it is more convenient to have control of the width of the \emph{dual} of a block than of the block itself. 
Given $r\geq 1$ and a block $\mathcal{G}_r$, 
we define the 
\textbf{dual width} of $\mathcal{G}_r$, denoted $\mathsf{Width}^\dagger(\mathcal{G}_r)$, to be length of the shortest path in the planar dual of $\mathcal{G}_r$ that starts and ends in the outside face, has its first and last edges in the left and right-hand boundaries of $\mathcal{G}_r$ respectively, and which does not visit the outside face other than at its endpoints. We call such a path a \textbf{dual left-right crossing} of $\mathcal{G}_r$. 

We claim that the dual width of $ \mathcal{G}_{r}$ is equal to the maximal size of a set of edge-disjoint (primal) paths from the bottom to the top of $\mathcal{G}_r$ (we call such a path a \textbf{primal bottom-top crossing}), whence its close connection to resistances. One such maximal set of primal bottom-top crossings can be found algorithmically by first taking the left-most primal bottom-top crossing, then the left-most primal bottom-top crossing that is edge-disjoint from the first one, and so on. The claim can be proved using the cut-cycle duality \cite[Theorem 14.3.1]{MR1829620} and Menger's theorem, but can also easily be  checked in our situation, see Figure~\ref{fig:dualblock} below.


\vspace{1em}
\begin{figure}[!h]
 \begin{center}
\hspace{2em} \includegraphics[width=9.75cm]{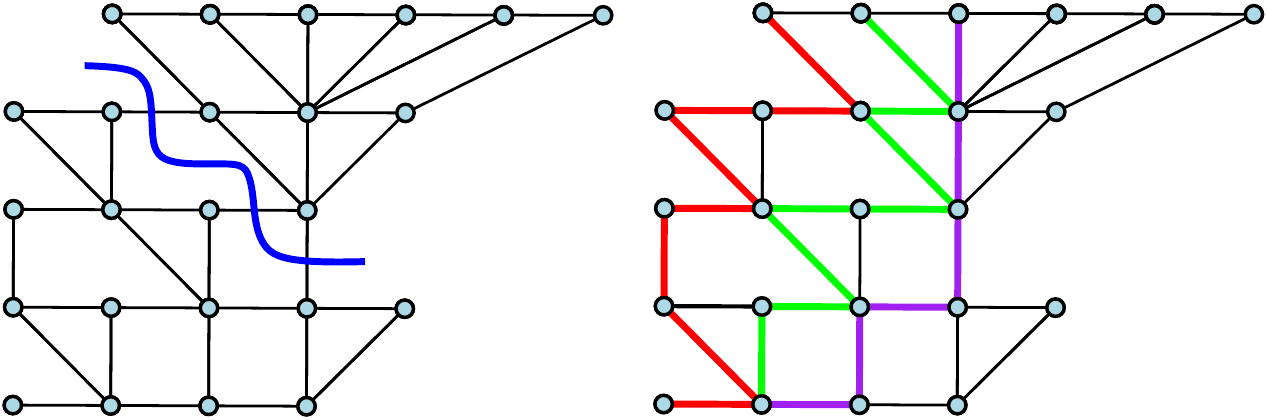}
 \caption{ 
 \small{\label{fig:dualblock} A block of height four that has width $2$ and dual width $3$. On the left is a dual left-right crossing of length three, on the right is a set of three edge-disjoint primal paths from the bottom to the top of the block. (Note that in general these paths might not have increasing heights as they do in this example.)}}
 \end{center}
 \vspace{-1em}
 \end{figure}

\noindent
For each $r\geq 1$, we define $g(r)$ to be the median dual width of $\mathcal{G}_r$, that is, the largest number such that \[ \mathbb{P}\big( \mathsf{Width}^\dagger( \mathcal{G}_{r}) \geq g(r)\big) \geq 1/2.\]
The proof of Theorem \ref{prop:renorm} goes through essentially unchanged to yield that there exists a constant $c>0$ such that
\begin{align*} g(r)  \geq c \cdot  \min \Big\{ m ;   (r/m)^{\beta} g(m) \Big\}
\qquad \forall 1 \leq m \leq c\cdot r,
\end{align*}
from which we obtain as before the following analogue of Theorem \ref{thm:main}.

\begin{theorem}[$g(r)$ is almost linear] \label{thm:maindual} If $\mu$ is critical and satisfies \eqref{eq:defstable} then there exists $c >0$ such that for all $r$ large enough we have  $$g(r) \geq c \,r.$$
On the other hand, if $\mu$ is critical and has finite non-zero variance then there exists $C>0$ such that for all $r$ large enough we have
$$g(r) \geq r \exp( - C \sqrt{ \log r}).$$
\end{theorem}

\subsection{The subadditive argument} \label{sec:subadditive}
In this section we suppose that $\mu$ is critical and has finite variance. We use the same notation as in the preceding section. 
 In particular, we let $\mathcal{Q_\infty}$ be the quarter-plane model, as defined in Section \ref{sec:blocks}, and recall that $\mathcal{Q}_\infty$ is indexed by $\{1,2,\ldots\}\times \{0,1,\ldots\}$. For each $m \geq n\geq 1$ let $\mathcal{Q}_{n,m}$ be the subgraph of $\mathcal{Q}_\infty$ induced by the trees $T_n,\ldots, T_m$, and let $L_{n,m}$ be the graph distance between $(n,0)$ and $(m,0)$ in $\mathcal{Q}_{n,m}$.
Our aim is to prove the following. 
\begin{proposition} \label{prop:subadditive}If $\mu$ is critical and has finite non-zero variance then 
$\lim_{n\to\infty}\frac{1}{n}L_{1,n}= 0$ a.s.
\end{proposition}
\proof The proof is based on a simple observation together with Kingman's subadditive ergodic theorem. 
 We clearly have that the random array $(L_{n,m})_{m \geq n \geq 1}$ is stationary in the sense that $(L_{n+k,m+k})_{m\geq n \geq 1}$ has the same distribution as $(L_{n,m})_{m\geq n \geq 1}$ for every $m\geq n \geq 1$ and $k\geq 1$, and is subadditive in the sense that $L_{n,m+k} \leq L_{n,m} + L_{m,m+k}$ for every $n,m,k \geq 1$. Moreover,
 since the $T_{i}$'s are i.i.d., the stationary sequence $((L_{n+k,m+k})_{m\geq n \geq 1})_{k\geq 0}$ is ergodic and we can apply Kingman's subadditive  ergodic theorem to deduce that   \begin{eqnarray} \label{eq:kingman} n^{-1}L_{1,n} \xrightarrow[n\to\infty]{a.s.} c,  \end{eqnarray} for some non-random constant $c\in [0,1]$. 

To finish the proof and show that $c=0$, we use the following observation. Recall that for $r \geq1$, we denoted by $\xi_{r}=:\xi_{r}^{(1)}$ the index of  the first tree among $T_{1}, T_{2}, \dots$ that reaches height $r$. We also denote by $\xi^{{(2)}}_{r}$ the index of the second such tree. Considering the path that starts at $(1,0)$, travels horizontally to $(\xi_{r}^{(1)},0)$, travels vertically up to the right-most element of $T_{\xi_r^{(1)}}$ in level $r$, takes one step to the right, and then travels vertically downwards to $(\xi_r^{(2)},0)$, as illustrated in Figure \ref{fig:shortcut}, yields the bound
 \begin{eqnarray} L_{1,\xi^{(2)}_{r}} \leq \xi_{r}^{(1)}+2r.   \label{eq:shortcut}\end{eqnarray}
 Using \eqref{eq:kolmogorov}, it is easy to show that $ r^{-1}\xi_{r}^{(1)}$ and $r^{-1} (\xi_{r}^{(2)}-\xi_{r}^{(1)})$ converge in distribution towards a pair of independent exponential random variables with the same parameter. In particular, it follows that
 \[
\liminf_{r\to\infty} \P\bigl( \xi_{r}^{(1)}+2r \leq  \varepsilon \cdot \xi_{r}^{(2)} \bigr)>0
 \]
 for every $ \varepsilon>0$. 
   This observation together with the a.s.\ convergence \eqref{eq:kingman} and the bound \eqref{eq:shortcut} yields that $c \leq \varepsilon$. Since this inequality is valid for every $ \varepsilon>0$ we must have that $c=0$.   \endproof

\begin{figure}[t]
 \begin{center}
 \includegraphics[width=8cm]{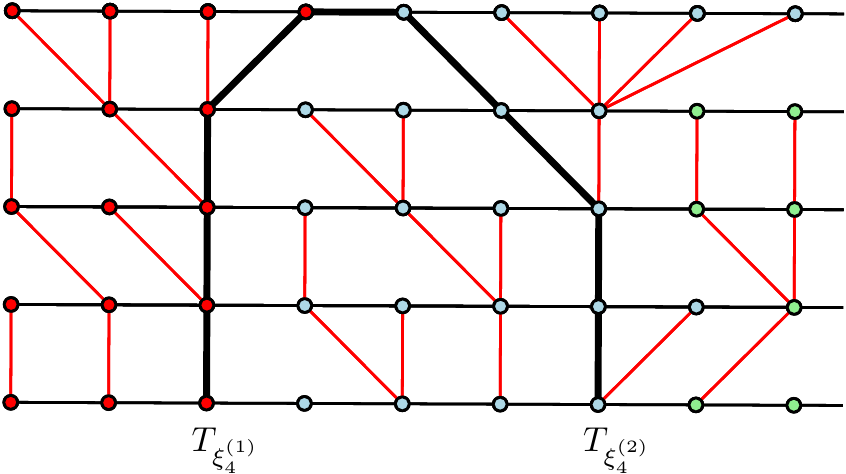}
 \caption{Illustration of the bound \eqref{eq:shortcut} \label{fig:shortcut}.}
 \end{center}
 \vspace{-0.5em}
 \end{figure}

 \section{Estimating the girth}
 \label{sec:width}
 In this section we will derive our Theorems \ref{thm:geometry} and \ref{thm:geometrystable} from the estimates on the geometry of blocks derived in the last section. 
  In order to do this, we relate the $\mu$-Galton--Watson tree conditioned to survive $T_\infty$ (and the graph $\Cinfty= \mathsf{Causal}( T_{\infty})$ obtained by adding the horizontal edges to $T_{\infty}$ as in Figure~\ref{fig:causal-arbre}) to the unconditioned quarter-plane model made of i.i.d.~$\mu$-Galton--Watson trees that we considered in Section \ref{sec:halfplane}. The main ingredient is the standard representation of $T_{\infty}$  using the spine decomposition \cite{LPP95b}, which we now review. We refer to \cite{AD13} for precise statements and proofs regarding this decomposition.\medskip

The plane tree $T_{\infty}$ has a unique spine (an infinite line of descent) which can be seen as the genealogy of a mutant particle which reproduces according to the biased distribution $\overline{\mu} =( k \cdot \mu_{k})_{k \geq 0}$ and from which exactly one of its offspring is chosen at random and declared mutant. All other particles reproduce according to the underlying offspring distribution $\mu$, see Figure \ref{fig:subforest} (left) and \cite{AD13} for more details. We deduce from this representation that, for every $n_{0} \geq 1$, conditionally on $\# \partial [T_{\infty}]_{n_{0}}$, if at generation $n_{0}$ we erase the only mutant particle and all its offspring then we obtain a forest of $(\# \partial[T_{\infty}]_{n_{0}}-1)$ independent $\mu$-Galton--Watson trees. We order the trees in this forest using the plane ordering of $T_{\infty}$ so that the first tree is the one immediately to the right of the spine and the last tree is the one immediately to the left of the spine. Denote this forest by $ \mathcal{F}_{n_{0}}$ and note that it can be empty.
We add the horizontal connections between inner vertices of $\mathcal{F}_{n_0}$ (except those linking the extreme vertices of a line) to get the graph $ \mathcal{C}_{n_{0}}$ which is then a subgraph of $ \Cinfty$. The graph $ \mathcal{C}_{n_{0}}$ truncated at height $k$ will be denoted by $[ \mathcal{C}_{n_{0}}]_{k}$. See Figure \ref{fig:subforest}.


\begin{figure}[!h]
 \begin{center}
 \includegraphics[width=14cm]{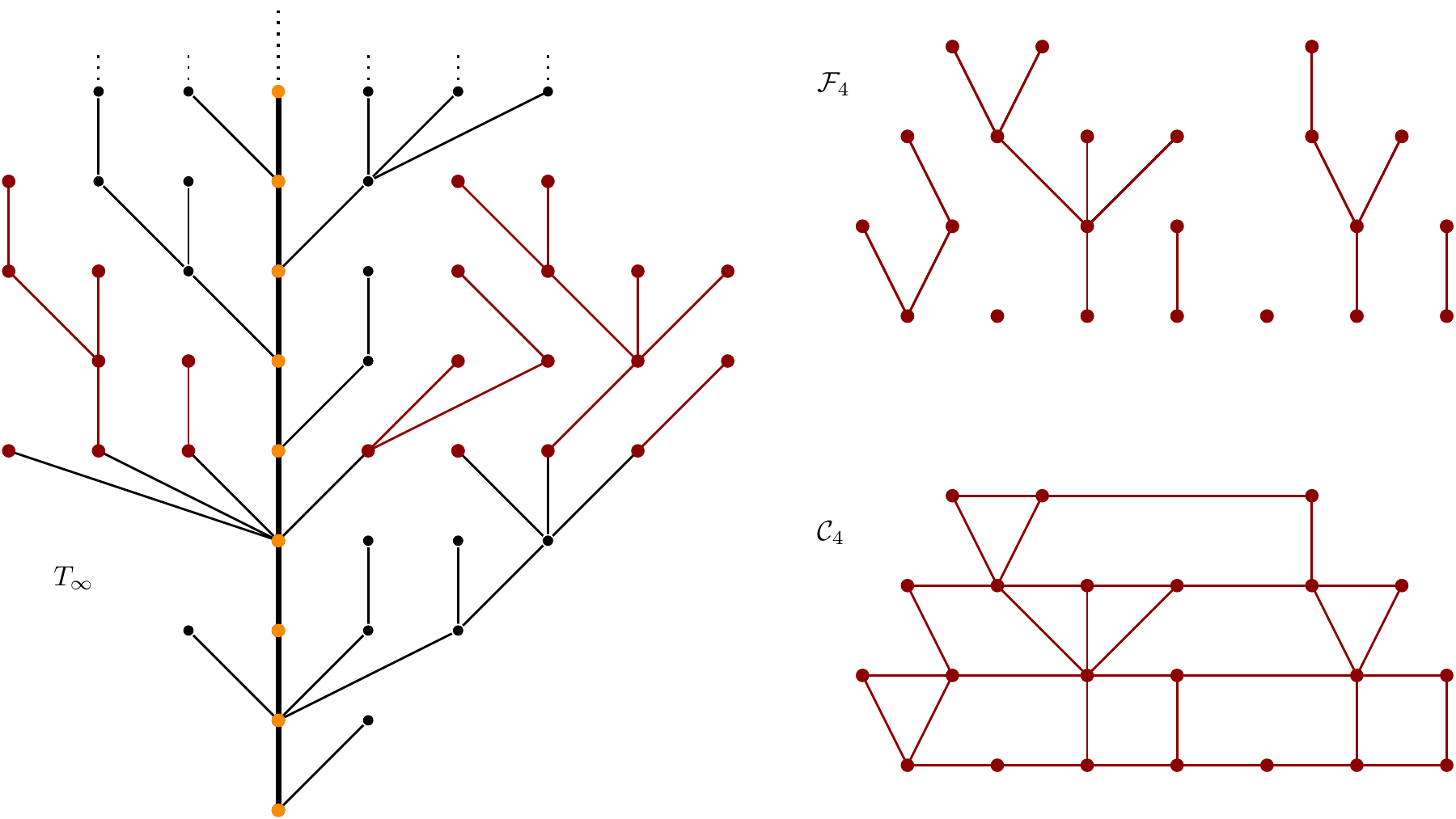}
 \caption{
 \small{Left: A piece of $T_{\infty}$ where the ancestral line of the mutant particles is highlighted. Right: the sub forest $ \mathcal{F}_{4}$ and its causal version $ \mathcal{C}_4$ obtained by adding the horizontal connections.  \label{fig:subforest}}
 }
 \end{center}
 \vspace{-2em}
 \end{figure}

 It is also standard that $T_{\infty}$ is the martingale biasing of the random variable $T$  by the non-negative martingale $( \# \partial [T]_{n})_{n \geq 0}$. That is, for every positive  function $f$ on the set of finite plane trees we have 
 $$ \mathbb{E}[f([T_{\infty}]_{n})] = \mathbb{E}[f([T]_{n}) \# \partial [T]_{n}]$$
 for every $n\geq 0$. In particular the size of the $n$-th generation $\# \partial [T_{\infty}]_{n}$ has the law of $\# \partial [T]_{n}$ biased by itself. It is also standard (see \cite[Theorem 4]{Pak76}) that $\# \partial [T_{\infty}]_{n}$ is of order $n^{\beta}$ (recall the definition of $\beta$ at the end of the Introduction) and more precisely that once rescaled by $n^{-\beta}$ it converges in distribution towards a positive random variable: \begin{eqnarray} n^{-\beta}\# \partial [T_{\infty}]_{n} \quad \xrightarrow[n\to\infty]{d} \quad \mathcal{X}'>0. \label{eq:yaglombis}  \end{eqnarray}
 Here again, the precise distribution of the random variable $ \mathcal{X}'$ will not be used. We shall however use a version of this estimate which is rougher for a given $n$ but holds simultaneously for all $n \geq 1$:
   \begin{lemma} \label{lem:lil} There is some positive contant $0<C<\infty$ such that  almost surely, for all $n$ large enough
   $$  n^\beta \big(\log n \big)^{-C} \leq \#\partial [T_{\infty}]_{n} \leq n^\beta \big(\log n \big)^{C}. $$
 \end{lemma}
 \proof We set $Z_{n}^{*} = \# \partial [T_{\infty}]_{n}$ and $Z_n = \# \partial [T]_{n}$ for an independent unconditioned $\mu$-Galton-Watson tree $T$, so that $(Z_n^*)_{n
\geq 0}$ is distributed as martingale biasing of $(Z_n)_{n
\geq 0}$ by itself. We first prove the lemma  along the subsequence $n= 2^{k}$. By \cite[Propositions 2.2 and 2.6]{CK08} there exist constants $c>0$ and $\delta >0$ such that 
 \begin{equation}\label{eq:CK08} \mathbb{P}\bigl( Z_{n}^{*} \notin (\lambda^{-1} n^{\beta }; \lambda n^{\beta })\bigr) \leq  c \lambda ^{-\delta}\end{equation}
for every $n\geq 1$ and $\lambda >1$. 
Putting $ n= 2^{k}$ and $\lambda = k^{ 2/\delta}$ we can use the Borel--Cantelli lemma to deduce that indeed   \begin{eqnarray}
k^{-2/\delta}2^{\beta k} \leq Z^{*}_{2^{k}} \leq k^{2/\delta}2^{\beta k}  \label{eq:2k}  \end{eqnarray}
for all sufficiently large $k$ almost surely.

 We now extend this estimate to all values $n \geq 1$, at the price of changing the exponent of the logarithm from $2/\delta$ to $C=(8/\delta) \vee 8$. We begin with the upper bound. Let $n\geq 1$ and let 
 \[
T_n = \inf\left\{m \geq n :  Z^{*}_{m} \geq m^{\beta} (\log m)^{C}\right\},
 \]
where we set $\inf \emptyset = \infty$. 
Since 
\[
\left\{Z^{*}_n \geq n^\beta (\log n)^{C} \text{ for infinitely many $n\geq 1$}\right\} = \bigcap_{n\geq 1}\left\{T_n < \infty\right\}
\]
It suffices to prove that $\lim_{n\to\infty} \P(T_n < \infty) =0$.
 Condition on the stopped $\sigma$-algebra $\mathcal{F}_{T_n}$, and let  $K_n = \lceil \log_2 T_n \rceil$. If $T_n<\infty$ then $ 4 T_n \geq 2^{K_n+1}- T_n \geq T_n$, and it follows from \eqref{eq:kolmogorov} and \eqref{eq:yaglom} that each of the $Z_{T_n}^* \geq T_n^{\beta} (\log T_n)^{C}$ particles in generation $T_n$ have conditional probability at least $c \, T_n^{-\beta}$ of having at least $T_n^{\beta}$ descendants at level $2^{K_n+1}$ for some constant $c>0$ (the one backbone particle having an even higher conditional probability). The conditional probability that this occurs for at least $\lceil c(\log T_n)^{C}/2 \rceil $ particles is uniformly positive by \eqref{eq:ASChernoff}, and we deduce that
 $$ 
  \mathbb{P}\Bigl( T_n<\infty \text{ and } Z^{*}_{2^{K_n+1}} \geq c'2^{\beta K_n} K_n^{C} \;\Big|\; \mathcal{F}_{T_n} \Bigr) \geq c'' {1}(T_n<\infty)$$
for some positive constants $c'$ and $c''$. Taking expectations over $\mathcal{F}_{T_n}$, we deduce that
\begin{align*}
c'' \P(T_n<\infty) \leq \mathbb{P}\Bigl( T_n<\infty \text{ and } Z^{*}_{2^{K_n+1}} \geq c'2^{\beta K_n} K_n^{C}  \Bigr) 
\leq \mathbb{P}\Bigl(Z^{*}_{2^{k+1}} \geq c'2^{\beta k} k^{C} \text{for some $k \geq \log_2 n$}\Bigr).
\end{align*}
The right hand side tends to zero as $n\to\infty$ by \eqref{eq:2k}, and we deduce that $\P(T_n<\infty) \to 0$ as $n\to\infty$ as claimed.

We now prove the lower bound, for which
 we adapt the proof of 
 \cite[Proposition 13]{BCsubdiffusive}.
  Notice that by	\eqref{eq:2k} we know that $ n^{-C/4} 2^{n\beta}\leq Z^{*}_{2^{n}} \leq 2^{n\beta} n^{C/4}$ eventually so we just need to prevent the process $Z^{*}$ from going down too low in-between times $2^{n}$ and $2^{n+1}$. For this, we consider  the conditional probability
\[ \mathbb{P}\left.\Bigl( 2^{n\beta}n^{-C/4} \leq Z_{2^{n+1}}^{*} \leq 2^{n\beta}n^{C/4} \mbox{ and } \exists 2^{n} < k <2^{n+1} \mbox{ s.t. } Z_{k}^{*} \leq 2^{n\beta}n^{-C} \ \right| Z_{2^{n}}^{*} = z_{0} \Bigr),\]
 for $n^{-C/4} 2^{n\beta}\leq z_{0} \leq 2^{n\beta} n^{C/4}$. Since the Markov chain $Z^{*}$ is the size-biasing of the chain $Z$ by itself (i.e., the $h$-tranform of $Z$ with respect to the function $h(n)=n$), this conditional probability is bounded above by
\begin{multline*}
\frac{2^{n\beta}n^{C/4}}{2^{n\beta}n^{-C/4}}
\\
\cdot\mathbb{P}\left.\Bigl(  2^{n\beta}n^{-C/4} \leq Z_{2^{n+1}} \leq 2^{n\beta}n^{C/4} \mbox{ and } \exists 2^{n} < k <2^{n+1} \mbox{ s.t. } Z_{k} \leq 2^{n\beta}n^{-C} \ \right| Z_{2^{n}} = z_{0}  \Bigr).
\end{multline*}
But now, since the process $Z$ is a non-negative martingale absorbed at $0$, the optional sampling theorem implies that the probability that the process $Z$ drops below $2^{n\beta}n^{-C}$ and then  later reaches a value larger than $2^{n\beta}n^{C/4}$ is less than $n^{-3C/4}$. Hence, the last display is bounded above by $n^{-C/4}$. Since $C\geq 8$ these probabilities are summable in $n$. Applying  Borel--Cantelli, we deduce that $Z_{n}^{*} \geq n^{\beta} (\log n)^{-C}$ eventually almost surely on the event that $n^{-C/4} 2^{n\beta}\leq Z^{*}_{2^{n}} \leq 2^{n\beta} n^{C/4}$ eventually. \endproof

A straightforward corollary of Lemma \ref{lem:lil} concerns the volume growth from the origin (as alluded to in the introduction): If $B_{r}$ denotes the graph ball of radius $r$ around in origin in $ \Cinfty$ and $\# B_{r}$ the number of vertices of $B_{r}$ then 
\begin{align} \label{eq:volumegrowth} r^{\beta+1} \big(\log r\big)^{-C} \leq \# B_{r}   \leq r^{\beta+1} \big(\log r \big)^{C} 
\end{align}
for all sufficiently large $r$ almost surely. 
 See \cite[Proposition 2.8]{BK06} and \cite[Lemma 5.1]{CK08} for more precise estimates.  We now have all the ingredients to complete the proof of our Theorems \ref{thm:geometry} and \ref{thm:geometrystable}. In fact, we will prove the following quantitative version of item $(i)$ of Theorem \ref{thm:geometry}.
\begin{proposition}[Quantitative girth lower bound for generic causal maps]  \label{prop:quantgeometry} Suppose $\mu$ is critical and has finite non-zero variance. Then there exists a constant $C$ such that
\[  \frac{1}{r} \mathsf{Girth}_{r}( \Cinfty) \geq  e^{-C \sqrt{\log r} } \]
almost surely for all $r$ sufficiently large. 
\end{proposition}

\proof[Proof of Theorem  \ref{thm:geometry} (i), Theorem \ref{thm:geometrystable} (i), and Proposition \ref{prop:quantgeometry}] Fix $ \varepsilon>0$. Pick  $n \geq 1$ large and consider the forest $ \mathcal{F}_{n}$ introduced above. This forest is obviously finite.

 By \eqref{eq:kolmogorov} and \eqref{eq:ASChernoff}, conditionally on the event that $\# \partial [T_{\infty}]_{n} \geq n^{\beta}(\log n)^{-C_{1}}$, the probability that there are at least $(\log n)^{2}$ trees  inside this forest which reach height at least $m=m(n)=n(\log n)^{-{(3+C_1)\over \beta}}$ (relative to their starting height of $n$) is lower bounded by 
 $$  \mathbb{P}\Big(\mathrm{Binomial}\Big(\lfloor n^{\beta}(\log n)^{-C_{1}} \rfloor , \mathbb{P}\big( \mathrm{Height}(T) \geq m(n)\big)\Big) \geq (\log n)^{2} \Big) \geq  1- \exp\bigl(- \delta (\log n)^3\bigr),$$ for some $\delta>0$. On this event, using our Theorem \ref{thm:main}, we see that the probability that $ [ \mathcal{C}_{n}]_{m}$ contains at least two disjoint blocks $\mathcal{G}^{(n,1)},\mathcal{G}^{(n,2)}$ of height $m$ whose widths are both at least 
\begin{align*} m e^{-C_2\sqrt{\log m}} &\geq n e^{-C_3\sqrt{\log n}} && \text{ if } \beta =1  \\
c m &= c n (\log n)^{-C_4} &&\text{ if } \beta >1
 \end{align*} is bounded from below by $1 - \exp(-\delta' (\log n)^2)$ for some finite constants $C_2,C_3,C_4$ and some $\delta' >0$. Thus, by Lemma \ref{lem:lil} and the Borel--Cantelli lemma we deduce that 
both events occur for all sufficiently large $m$ almost surely. 

Let $n+m/4 \leq \ell \leq n+3m/4$, and let $u,v$ be vertices at height $\ell$ that are in the left and right boundaries of the block $\mathcal{G}^{(n,1)}$ respectively. Then any path from $u$ to $v$ must either leave the strip of vertices of heights between $n$ and $n+m$, or else must cross at least one of the blocks $\mathcal{G}^{(n,1)}$ or $\mathcal{G}^{(n,2)}$. From here we see immediately that there exists $C_5<\infty$ such that the bound
\begin{align*} \mathsf{Girth}_{\ell}(\Cinfty) &\geq \begin{cases}
n e^{-C_5\sqrt{\log n}}
& \text{ if }\beta =1\\
n (\log n)^{-C_5}& \text{ if } \beta>1,
\end{cases}
& n+m/4\leq \ell \leq n+3m/4
\end{align*}
holds  eventually for all sufficiently large $n$ almost surely,
concluding the proof.\endproof

The second part of Theorem \ref{thm:geometrystable} follows the same lines. Let us sketch the argument.
\proof[Sketch of proof of the second part of Theorem \ref{thm:geometrystable}] Suppose that $\mu$ satisfies \eqref{eq:defstable} and recall that $\beta = \frac{1}{\alpha-1}$ >1.  Since by \eqref{eq:yaglombis} the variable $\# \partial [T_{\infty}]_{n}$ is typically of order $n^{\beta}$, using \eqref{eq:kolmogorov} we deduce that the number of trees in the forest $ \mathcal{F}_{n}$ which reach height $ \eta n$ tends in probability to $\infty$ as $ \eta \to 0$ and $n \to \infty$. In particular, for any $ \varepsilon>0$ and  any $k_{0} \geq 1$ we can find $\eta$ small enough such that for any large enough $n$, the graph $\mathcal{C}_{n}$ contains at least $k_{0}$ independent blocks of height $\eta n$ with probability at least $1 - \varepsilon$. By Theorem \ref{thm:main}, with probability at least $1- (1+2k_0)2^{-k_{0}}$ the left-right width of at least two of these blocks is larger than $c \eta n$. Choosing $k_{0}$ so that $(1+2k_0)2^{-k_{0}}\leq \varepsilon$, we deduce (using the same argument as above) that with probability at least $1 - 2 \varepsilon$ the girth at level of $ \Cinfty$ between $n(1+\eta/4)$ and $n(1+3\eta/4)$ is at least $c\eta n$. This entails Theorem \ref{thm:geometrystable}.\endproof

Finally, we prove the upper bound on the girth in the finite-variance case.

\proof[Proof of Theorem \ref{thm:geometry} (ii)] We fix $\mu$ critical and having a finite, non-zero variance. Fix $n \geq 1$ large and consider the graph $ \mathcal{C}_{n}$, which is a subgraph of $ \Cinfty$. As before, conditional on $ \# \partial [T_{\infty}]_{n} = \ell $ this graph is made of $\ell-1$ i.i.d.~Galton--Watson trees together with the added horizontal connections.  Proposition \ref{prop:subadditive} directly tells us that the distance inside $ \mathcal{C}_{n}$ between its bottom-left corner $x$ and its bottom-right corner $y$ is $o(\ell)$ with high probability. 

\medskip

\noindent \begin{minipage}{9cm}
Since $x$ and $y$ are both incident to the spine vertex at level $n$, we can use two horizontal edges  to link $x$ to $y$ in $ \Cinfty$ as depicted on the right. Since $\ell=O(n)$ with high probability by \eqref{eq:yaglombis},
this argument shows that for each $\eps>0$ there exists $N<\infty$ such that if $n\geq N$ then we can construct, with probability at least $1-\eps$,  a loop $ \mathcal{L}$ inside $  \Cinfty$ of length at most $\eps n$ that separates $\rho$ from $\infty$  and that only contains vertices of height between $n$ and $(1+\eps)n$. 
\end{minipage} 
\hspace{1cm}
\begin{minipage}{3cm}
\includegraphics[width=3cm]{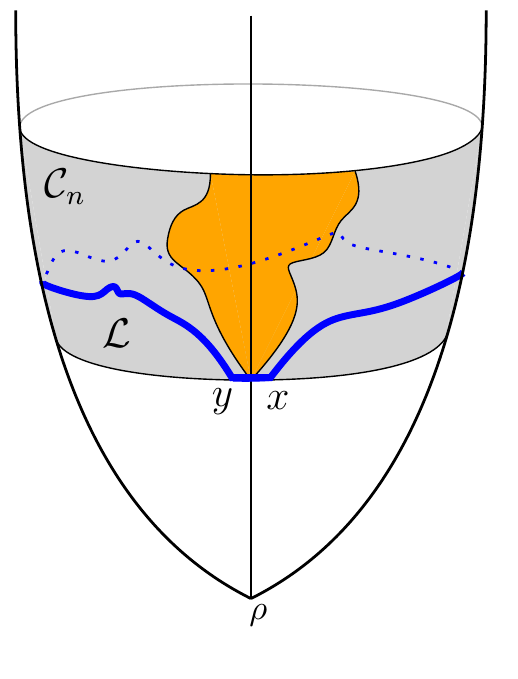}

\end{minipage}
\medskip

Let $n\geq N$, condition on this event, and consider the set of vertices of $ \Cinfty$ at height $n'=n+\lceil \eps n \rceil$. Each such vertex is connected to a vertex at height $n$ by a path of length $\lceil \eps n\rceil$, and this path must intersect $\mathcal{L}$. We deduce that the distance between any vertex of $[T_{\infty}]_{n'}$ and $ \mathcal{L}$ is less than $ \varepsilon n$. We deduce that $\mathsf{Girth}_{n'}(T_{\infty}) \leq 3\eps n$ with probability at least $1-\eps$ for every $n\geq N$, from which the proof may easily be concluded. \endproof

 \section{Resistance growth and spectral dimension}
 \label{sec:spectral}

In this section we will prove Theorem \ref{thm:spectral}, via Theorem \ref{thm:resistance}. Since certain arguments are valid in general, we highlight when finite variance is needed. 
\subsection{Resistance}

 The resistance will be controlled through the method of random paths and builds upon the geometric estimates established in the preceding section. In this section, all resistances will be taken with respect to the graph $\Cinfty$. Before diving into the proof  of Theorem \ref{thm:resistance}, we first prove that $ \Cinfty$ is recurrent (i.e., that $ \mathsf{R_{eff}}(\rho \leftrightarrow B_r^c)=\mathsf{R_{eff}}( \rho \leftrightarrow \partial [T_\infty]_{r+1} ; \Cinfty) \to \infty$ as $r\to\infty$).
 \begin{proposition} \label{prop:recurrence} If $\mu$ has finite non-zero variance then $ \Cinfty$ is recurrent almost surely.
 \end{proposition}
 \proof We apply the Nash--Williams criterion for recurrence \cite[(2.14)]{LP10}, using the obvious collection of cut-sets given by the sets of edges linking level $r$ to level $r+1$ for each $r\geq1$. This edge set has cardinality precisely $\#\partial [T_{\infty}]_{r+1}$  so the proposition reduces to checking that
 $$ \sum_{r=1}^{\infty} \frac{1}{ \# \partial [T_{\infty}]_{r}} = \infty, \quad \mbox{ almost surely}.$$
Since we have that $\frac{1}{ \# \partial [T_{\infty}]_{r}} = \frac{1}{r} \frac{r}{ \# \partial [T_{\infty}]_{r}}$ and  by \eqref{eq:yaglombis} that the random variable $\frac{r}{ \# \partial [T_{\infty}]_{r}}$ converges in law towards a positive random variable, the last display is a direct consequence of Jeulin's Lemma \cite[Proposition 4 c]{Jeu82}. \endproof

\begin{remark}
Let us briefly discuss quantitative resistance lower bounds. 
It follows immediately from Nash-Williams that
\[
\E \left[\mathsf{R_{eff}}(\rho \leftrightarrow B_r^c) \right] \geq \sum_{k=1}^r \frac{1}{k} \E \left[\frac{k}{ \# \partial [T_{\infty}]_{k}}\right] \geq c \log r 
\]
for some $c>0$ by \eqref{eq:yaglombis}.
 With a little further effort 
 one can prove an almost sure lower bound on the resistance growth of the form 
\[
\mathsf{R_{eff}}(\rho \leftrightarrow B_r^c) \geq \sum_{k=1}^r \frac{1}{ \# \partial [T_{\infty}]_{k}} \geq c \log r \qquad \text{ for all $n$ sufficiently large a.s.}
\] 
Indeed, in the analogous statement for the CSBP the contributions from successive  dyadic scales form a stationary ergodic sequence and the result follows from the ergodic theorem. Pushing this argument through to the discrete case requires one to handle  some straightforward but tedious technical details. 
We do not pursue this further here.
\end{remark}

\proof[Proof of Theorem \ref{thm:resistance}.] 

Recall the assumption that $\mu$ is critical and has finite, non-zero variance. By Lemma \ref{lem:lil} there exists a constant $C_1>0$ such that the number of vertices in level $r$ satisfies $r (\log r)^{-C_{1}} \leq \# \partial [T_{\infty}]_{r} \leq r (\log r)^{C_1}$ for all $r$ larger than some almost surely finite random $r_0$. 
Arguing as in the proof of Proposition \ref{prop:quantgeometry} but applying Theorem \ref{thm:maindual} instead of Theorem \ref{thm:main}, we obtain that there exist constants $C_2,C_3$ such that there exists an almost surely finite $m_0$ such that for every $m\geq m_0$ and every $m (\log m)^{-C_2} \leq k\leq 3m (\log m)^{-C_2}$, 
the subgraph $[\mathcal{C}_{m}]_{k}$ of $\Cinfty$, defined in Section \ref{sec:width}, contains a block of height equal to $k$ and dual width at least $k e^{-C_3\sqrt{\log k}}$.

Consider the increasing sequence of natural numbers $h_n$ defined by
\[
h_n = \left\lfloor 
\exp\left((C_2+1)^{1/(C_2+1)}n^{1/(C_2+1)}\right)
\right\rfloor,
\]
and let $k_n=h_{n+2}-h_n$.
These numbers have been chosen to satisfy the asymptotics $k_n \sim 2 h_n (\log h_n )^{-C_2}$ as $n\to\infty$, so that in  particular $ h_n (\log h_n )^{-C_2} \leq k_n \leq 3 h_n (\log h_n )^{-C_2}$ for all $n$ larger than some $n_0'<\infty$. Thus, it follows from the discussion in the previous paragraph that there exists an almost surely finite $n_0''$ such that for each $n\geq n_0''$, the subgraph $[\mathcal{C}_{h_n}]_{k_n}$ of $\Cinfty$ contains a block $\mathcal{G}^{(n)}$ of height equal to $k_n$ and dual width at least 
$k_n e^{-C_3\sqrt{\log k_n}}$. 
Let $n_0 \geq n''_0$ be minimal such that $h_{n_0} \geq r_0$ and let $\Omega$ be the almost sure event that $n_0$ is finite.


Since the resistance $\mathsf{R_{eff}}(\rho \leftrightarrow B_{r}^{c})$ is increasing in $r$, it suffices to prove that there exists a constant $C_4$ such that
\[\mathsf{R_{eff}}\bigl(\rho \leftrightarrow B_{h_n}^{c}\bigr) \leq e^{C_4\sqrt{\log h_n}}\]
almost surely as $n\to\infty$. We will prove that this is the case deterministically on the event $\Omega$. In order to do this,  we use the \emph{method of random paths} (see \cite[Chapter 2.5]{LP10}). In particular, we will construct   a random path $\Gamma$ from $\rho$ to the boundary of the ball of radius $h_n$, and 
then bound the resistance by 
the ``energy'' of the path\footnote{Strictly speaking the quantity on the right of \eqref{eq:energy} is not the energy of $\Gamma$, but rather an upper bound on the energy of $\Gamma$.}:
  \begin{eqnarray} \label{eq:energy} \mathsf{R_{eff}}(\rho \leftrightarrow B_{r}^{c}) \leq 2 \sum_{ e \in \mathsf{Edges}(B_{r})} \mathbb{P}( \Gamma \mbox{ goes through }e \mid \Cinfty  \mbox{ and } \Omega)^{2}.  \end{eqnarray}





Condition on $\Cinfty$ and the event $\Omega$.
By the discussion of Section \ref{sec:dualdiam}, for each $m \geq n_0$, the subgraph $\mathcal{G}^{(m)}$ of $[\mathcal{C}_{h_m}]_{k_m}$ contains a set of at least $k_m e^{-C_3\sqrt{\log k_m}}$ edge-disjoint paths linking its bottom boundary to its top boundary. Indeed, the maximal size of such a set is equal to the dual left-right width of $\mathcal{G}^{(m)}$. Fix one such maximal set for each $m$ and let $\Gamma^{(m)}$ be a uniformly chosen element of this set. We let $s_{n_0}=h_{n_0}$ and for each $m> n_0$ we let $s_m$ be a uniform index between $h_m$ and $h_{m+1}$.  

We now build the random path $\Gamma$ starting from $\rho$ inductively as follows.  To start, we pick arbitrarily a path from $\rho$ to level $h_{n_0}$ to be the initial segment of $\Gamma$. We then let $\Gamma$ travel horizontally around level $h_{n_0}$ to meet the starting point of the path $\Gamma^{(n_0)}$. Following this, for each $n_0 \leq m\leq n-3$,
 between heights $s_{m}$ and $s_{m+1}$, the path $\Gamma$ follows the segment of $\Gamma^{(m)}$ between its last visit to height $s_m$ and its first visit to height $s_{m+1}$. When $\Gamma$ reaches level $s_{m+1}$, it travels horizontally around that level to join the path $\Gamma^{(m+1)}$ at the site of its last visit to that level. Finally, $\Gamma$ takes the segment of $\Gamma^{(n-2)}$ between levels $s_{n-2}$ and $h_{n}$, at which point it stops.

 \vspace{1em}
\begin{figure}[!h]
 \begin{center}
 \includegraphics[width=12cm]{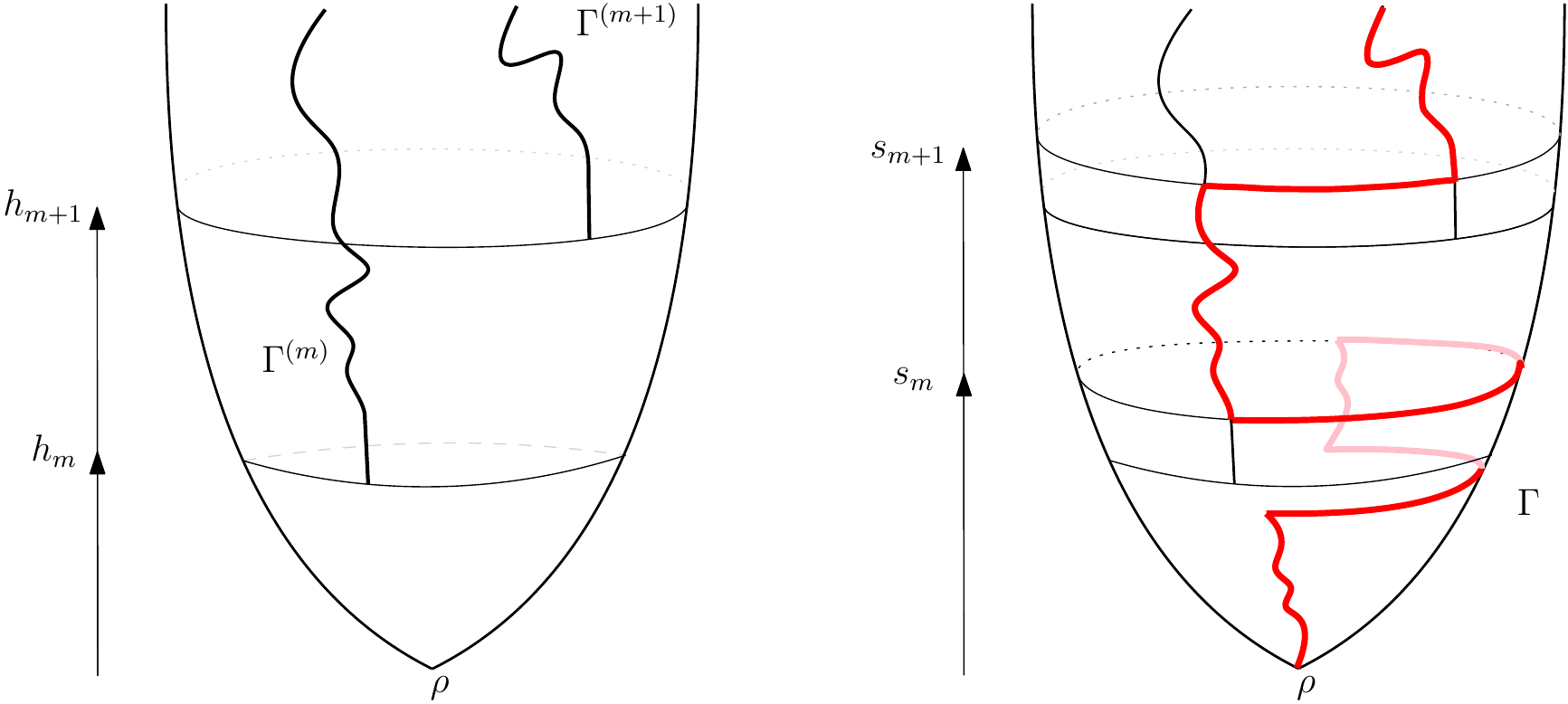}
 \caption{ \label{fig:randompath}
 \small{
 Illustration of the random path $\Gamma$ used in the proof of Theorem \ref{thm:resistance}. 
Left: for each $n_0 \leq m \leq n-2$ we have a path $\Gamma^{(m)}$ from level $h_m$ to level $h_{m+2}$. Right: The path $\Gamma$ switches from $\Gamma^{(m)}$ to $\Gamma^{(m+1)}$ by turning through a horizontal arc at the random height $s_{m+1}$.
 }}
 \vspace{-1em}
 \end{center}
 \end{figure}
We shall now estimate the energy of this random path.
Let $e \in B_{h_{n}}$ be an edge of height $h_m \leq \ell < h_{m+1}$ for some $n_0 \leq m < n$, where the height of an edge is defined to be the minimal height of its endpoints. Then we can compute that 
 \begin{multline*}
\P( \Gamma \text{ goes through } e \mid \Cinfty, \Omega)  \\
\leq\P( \Gamma^{(m-1)} \text{ goes through } e \mid \Cinfty, \Omega) + \P( \Gamma^{(m)} \text{ goes through } e \mid \Cinfty, \Omega) + \P(s_m = \ell)\\
\leq k_{m-1}^{-1} e^{C_3 \sqrt{\log k_{m-1}}} + k_{m}^{-1} e^{C_3 \sqrt{\log k_{m}}} + \frac{1}{h_{m+1}-h_m} \leq \ell^{-1} e^{C_5 \sqrt{\log \ell}} \, ,
 \end{multline*}
 where $C_5>0$ is another constant. Note that the number of edges at height $\ell$ is equal to $\#\partial [T_\infty]_\ell + \#\partial [T_\infty]_{\ell+1}$, and hence is at most $O(\ell \log^{C_1} \ell)$ on the event $\Omega$. 
 On the other hand, the initial segment of $\Gamma$ reaching from $\rho$ to level $h_{n_0}$ increases the energy of $\Gamma$ by at most a constant. Thus, we have that
 \[
\mathsf{R_{eff}}\bigl(\rho \leftrightarrow B^c_{h_n}\bigr) \leq O(1) + 
\sum_{\ell=h_{n_0}}^{h_n} \left[\ell^{-1} e^{C_5\sqrt{\log \ell}} \right]^2 O(\ell \log^{C_1} \ell) \leq e^{C_6 \sqrt{\log h_n}} \, ,
 \]
 for some constant $C_6>0$, as claimed.
\endproof

\begin{remark} One can adapt the proof of Theorem \ref{thm:spectral} to the $\alpha$-stable case. 
Following the same construction of the random path as in the above proof and applying Lemma \ref{lem:lil} with the appropriate $\beta >1$ we now deduce that the energy of the path $\Gamma$ linking $\rho$ to $B_{r}^{c}$ is of order
\[ \mathsf{R_{eff}}(\rho \leftrightarrow B_{r}^{c})  \leq \sum_{h=1}^{r} h^{-2} h^{\beta} \big(\log h\big)^{C_{1}} \leq r^{\beta-1} \big( \log r \big)^{C_{2}}\]
for some $C_{1},C_{2}>0$ as $r\to\infty$.
In particular, the resistance exponent $\mathfrak{r}$ (if it is well-defined) satisfies $ \mathfrak{r} \leq \beta-1 = \frac{2- \alpha}{\alpha-1}$. However, this  bound on the resistance becomes trivial in the regime $\alpha \in(1;3/2]$  since then $\beta \geq 2$ and we obtain a super-linear upper bound on the resistance... which is trivially at most $r$!
\end{remark}


\subsection{Spectral dimension and diffusivity (Theorem \ref{thm:spectral})}

We can now prove Theorem \ref{thm:spectral}. 
\proof[Proof of Theorem \ref{thm:spectral}] 
 We denote by $P^{n}(x,y)$ the $n$-step transition probabilities of the simple random walk $(X_n)_{n\geq0}$ on the graph $\Cinfty$. Recall that $ \mathrm{P}_{ \Cinfty,\rho}$ is the law of the random walk on $ \Cinfty$ started from $\rho$ and recall also that $B_{r}$ denote the ball of radius $r$ around the origin vertex $\rho \in \Cinfty$. We will split the proof of Theorem \ref{thm:spectral} into lower and upper bounds for return probabilities and typical displacements. As we will see, the upper bound for the return probability is a simple consequence of our resistance estimates (Theorem \ref{thm:resistance}) while the upper bound on the typical displacement is a standard application of the Varopoulos--Carne heat kernel bound for polynomially growing graph. Let us proceed. 
\medskip

\noindent (\textbf{Return probability upper bound.}) Recall that $\deg(\rho) \mathsf{R_{eff}}\bigl(\rho \leftrightarrow B_r^c \bigr)$ is equal to the expected number of times that the random walk started at $\rho$ visits $\rho$ before first leaving $B_r$. By the spectral decomposition for reversible Markov chains (see \cite[Lemma 12.2]{MCMT2E}) we know that $P^{2n}(x,x)$ is a decreasing function of $n$ for every vertex $x$ of $\Cinfty$.  
Hence letting $\tau_r$ be the first time that the random walk visits $B_r^c$,  we have the bound
\[
(n+1) P^{2n}(\rho,\rho)\leq \sum_{m=0}^n P^{2m}(\rho,\rho) \leq \mathbf{E}_{\Cinfty,\rho}\left[\sum_{m=0}^{\tau_{2n}} {1}\bigl( X_m = \rho\bigr) \right] = \deg(\rho)\Reff\bigl(\rho \leftrightarrow B_{2n}^c \bigr).
\]
Thus, applying Theorem \ref{thm:resistance} yields that
\begin{equation}
\label{eq:Pupper}
P^{2n}(\rho,\rho) \leq n^{-1} e^{C\sqrt{ \log n}}
\end{equation}
for all sufficiently large $n$ almost surely.

To obtain a similar bound for odd $n$, we 
use the well-known fact that return probabilities are log-convex in the sense that $P^{n+m}(\rho,\rho) \leq \sqrt{P^{2n}(\rho,\rho)P^{2m}(\rho,\rho)}$ for every $n,m\geq 0$ \cite[Lemma 3.20]{aldous-fill-2014}: 
Applying this fact together with \eqref{eq:Pupper} we obtain that
\begin{equation}
\label{eq:Pupper2}
P^{n}(\rho,\rho) \leq \sqrt{P^{2\lfloor n/2 \rfloor}(\rho,\rho)P^{2\lceil n/2 \rceil}(\rho,\rho)} \leq 2 n^{-1} e^{C\sqrt{ \log n}}
\end{equation}
for all sufficiently large $n$ almost surely.

\medskip
\noindent 
(\textbf{Typical displacement upper bound.})
 Recall the classical Varopulous-Carne bound \cite[Section 13.2]{LP10}, which implies that for every vertex $x$ of $\Cinfty$ and every $n\geq 1$ we have that
\[ P^n(\rho,x) \leq 2\sqrt{\frac{\deg(x)}{\deg(\rho)}} \exp\left[- \frac{1}{2n}\mathrm{d}^{\Cinfty}_\mathrm{gr}(\rho,x)^2 \right]  \leq 2\deg(x) \exp\left[- \frac{1}{2n}\mathrm{d}^{\Cinfty}_\mathrm{gr}(\rho,x)^2 \right].\]
Observe that, since every vertex of $\Cinfty$ at height $n\geq 1$ has at most three edges emanating from it that connect to vertices at height less than or equal to $n$, we have that  $ \sum_{x \in B_n} \deg(x) \leq 4 \cdot \# B_{n+1}$. Thus, it follows from \eqref{eq:volumegrowth} that there exists a constant $C$ such that  $\sum_{x \in B_n} \deg(x) \leq n^2 \big(\log n\big)^{C}$ for all sufficiently large $n$ almost surely. By the last display $\sum_n  \mathrm{P}_{ \Cinfty, \rho}(X_n \notin B_{\sqrt{5 n \log n}} )<\infty$ for almost all realizations of $ \Cinfty$. It follows by Borel--Cantelli (under $ \mathrm{P}_{\Cinfty,\rho}$) that
\begin{equation}
\label{eq:VC}
\limsup_{n\to\infty} \frac{ \mathrm{d}^{\Cinfty}_\mathrm{gr}(\rho,X_n)}{\sqrt{n\log n}} \leq \sqrt{5}
\end{equation}
almost surely.  This gives one side of the claim that $\nu=1/2$.

\medskip
\noindent (\textbf{Return probability lower bound}.) To get a lower bound on $P^{n}(\rho,\rho)$, first observe that
\begin{align*}P^{2n}(\rho,\rho) = \sum_{x\in V} P^n(\rho,x) P^n(x,\rho) = \sum_{x\in V} \frac{\deg(\rho)}{\deg(x)} P^n(\rho,x)^2.
\end{align*}
It follows that, for every $r\geq 0$,
\begin{align}
P^{2n}(\rho,\rho) \geq \sum_{x\in B_r} \frac{\deg(\rho)}{\deg(x)} P^n(\rho,x)^2 \geq 
 \deg(\rho) \mathbf{P}_{\Cinfty,\rho}\bigl(X_n \in B_r\bigr)^2  \frac{1}{\sum_{x \in B_r} \deg(x)}
\label{eq:dsleq2nugproof}
\end{align}
where the second inequality follows by Cauchy--Schwarz. 
Taking $r=\lfloor \sqrt{5 n \log n} \rfloor$  we deduce by \eqref{eq:volumegrowth} and the above application of Varopoulos--Carne that there exists a constant positive $C$ such that
\begin{align*}
P^{2n}(\rho,\rho) &\geq \deg(\rho) \mathbf{P}_{\Cinfty,\rho}\bigl(X_n \in B_{r} \bigr)^2 \left(4 \# B_{r +1 }\right)^{-1} \\& \geq  \deg(\rho) \bigl(1-o(1)\bigr)n^{-1} \big(\log n \big)^{-C}  \geq  n^{-1}\big(\log n\big)^{-2C}
\end{align*}
almost surely as $n\to\infty$. Together with \eqref{eq:Pupper} this implies that $d_s(\Cinfty)$ exists and equals $2$ a.s. 

\medskip
\noindent (\textbf{Typical displacement lower bound.})
Finally, to bound the probability that the displacement of the random walk is smaller than $n^{1/2-\eps}$, we rearrange \eqref{eq:dsleq2nugproof} and apply \eqref{eq:volumegrowth} and \eqref{eq:Pupper2} to deduce that there exists a constant $C$ and some almost surely finite  $n_0$ and $r_0$ such that
\[
 \mathbf{P}_{\Cinfty,\rho}\bigl( X_n \in B_r\bigr)^2 \leq r^2 n^{-1} e^{C\sqrt{\log n}} \big( \log r\big)^{C}
\]
for every $n\geq n_0$ and $r\geq r_0$, and it follows immediately that
\[
\lim_{n\to\infty} \mathbf{P}_{\Cinfty,\rho}\Bigl( X_n \in B_{n^{1/2-\eps}}\Bigr) =0
\]
for every $\eps>0$ a.s. Together with \eqref{eq:VC} this implies that $\nu(\Cinfty)$ exists and equals $1/2$ a.s.
 \endproof

\section{Extensions and comments} \label{sec:comments}

\subsection{Back to Causal Triangulations} \label{sec:causaltrig}
\begin{definition} A causal triangulation is a finite rooted triangulation of the sphere such that the maximal distance to the origin of the map is attained by a single vertex, and for each $k\geq 0$ the subgraph induced by the set of vertices at distance $k \geq 0$ from the origin is a cycle.
\end{definition}
In this work, we focused on the model $ \mathsf{Causal}(\tau)$ which is obtained from a plane tree $\tau$ by adding the horizontal connections between vertices are the same generation. As explained in Figure \ref{fig:causal-arbre}, to get a causal triangulation one needs also to  triangulate the faces from their top right corners. (Furthermore, one must add a point at the top of the graph to triangulate the top most face, even if this face is already a triangle). As explained in \cite{DJW10} this construction is a bijection between the set of finite rooted plane trees and the set of finite causal triangulations.
%

When this procedure is applied to the uniform infinite random tree $ \mathbb{T}_{\infty}$ (which is distributed as a critical geometric Galton-Watson tree conditioned to survive forever) the resulting map $ \mathsf{CauTri}( \mathbb{T}_{\infty})$ is the uniform infinite causal triangulation (UICT) as considered in \cite{DJW10,SYZ13}. The large scale geometries of $ \mathsf{CauTri}( {T}_{\infty})$ and $\mathsf{Causal}( {T}_\infty)$ are very similar and it is easy to adapt the results of the present paper to this setting.

Moreover, while it is certainly possible to simply run our arguments again to analyze $\mathsf{CauTrig}(T_\infty)$ instead of $\mathsf{Causal}(T_\infty)$, it is also possible to simply \emph{deduce} 
 versions of each of our main theorems concerning $\mathsf{CauTri}(T_\infty)$    from the statements that we give. Indeed, using the fact that the largest face in the first $n$ levels of $\mathsf{Causal}(T_\infty)$ 
  is at most logarithmically large in $n$ yields that distances within the first $n$ levels of $\mathsf{CauTrig}(T_\infty)$ are smaller than those in $\mathsf{Causal}(T_\infty)$ by at most a logarithmic factor. Moreover, an analogue of the resistance upper bound of Theorem \ref{thm:resistance} follows immediately since $\mathsf{Causal}(T_\infty)$ is a subgraph of $\mathsf{CauTri}(T_\infty)$.

   We let the reader stare at the two beautiful pictures in Figure \ref{fig:beauty}.

\begin{figure}[!h] 
 \begin{center} 
 \includegraphics[width=6.5cm]{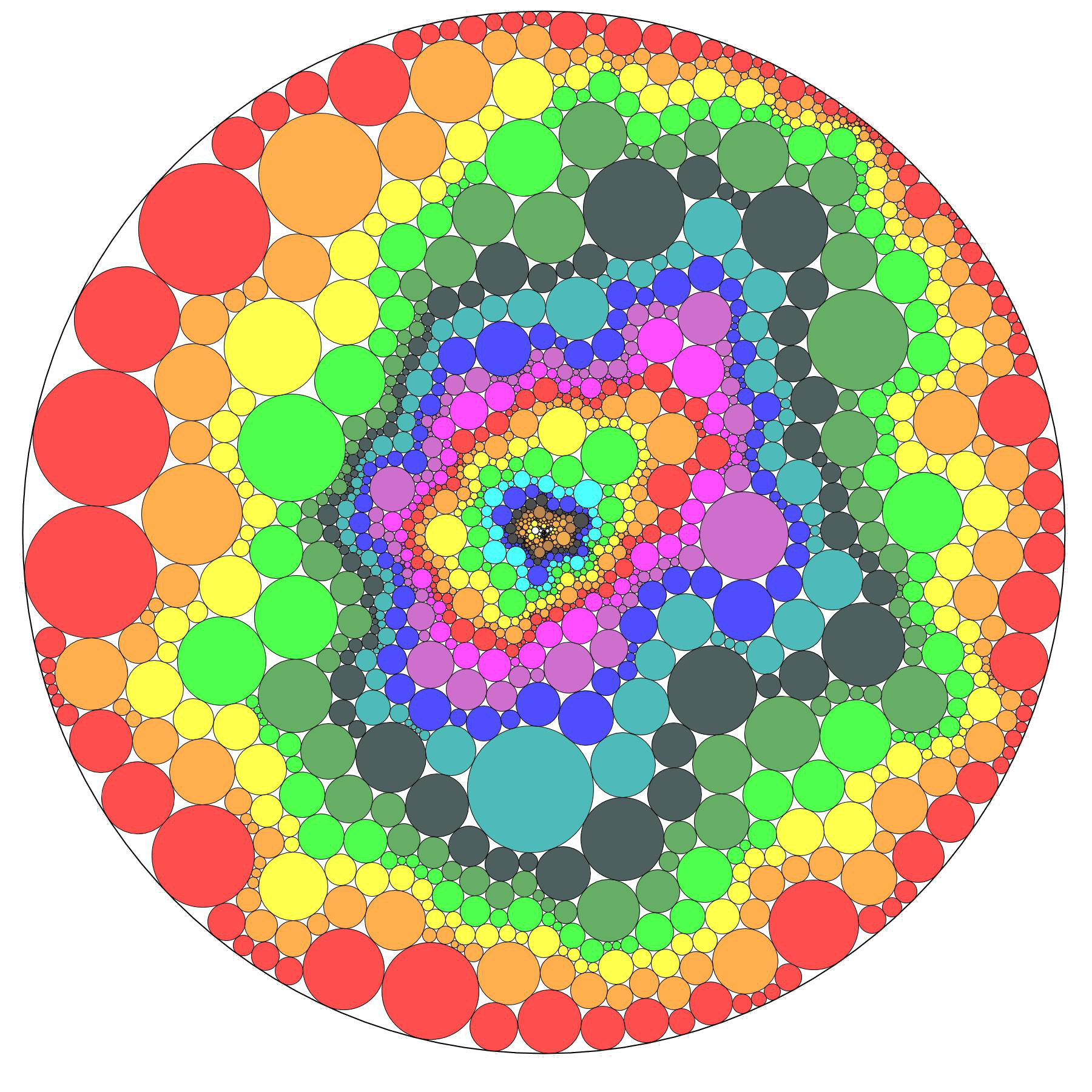}\hspace{1cm}  \includegraphics[width=6.5cm]{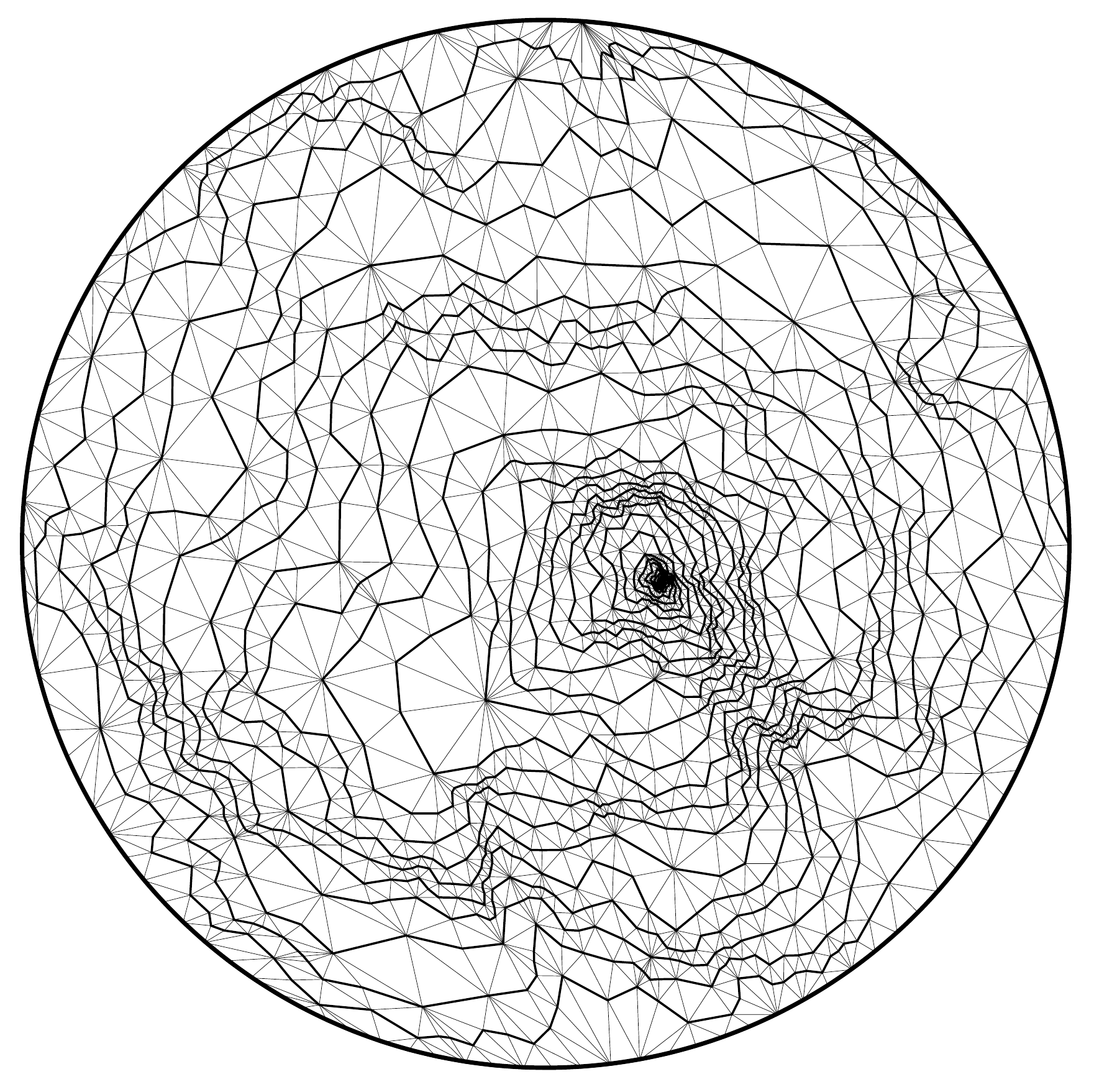}
 \caption{Two discrete uniformizations of a large ball in two different random causal triangulations. On the left, using the circle packing (so that the origin is at the center); on the right using Tutte's barycentric embedding (so that the vertices on the boundary are evenly spread). The generations are represented with colors on the left and using a spiral contour on the right. The circle packing was generated by Thomas Budzinski using Ken Stephenson's software.}
 \label{fig:beauty}
 \end{center}
 \end{figure}

\subsection{Causal carpets}
\label{sec:carpet}
Finally, we want to stress that our results can be adapted to various other graphs obtained from trees by ``adding the horizontal connections''. For example one could decide, when transforming a plane tree $\tau$ to add the horizontal connections but only keeping the extreme most edges of each branch point, see Figure \ref{kri1}.
\begin{figure}[!h]
 \begin{center}
 \includegraphics[width=10cm]{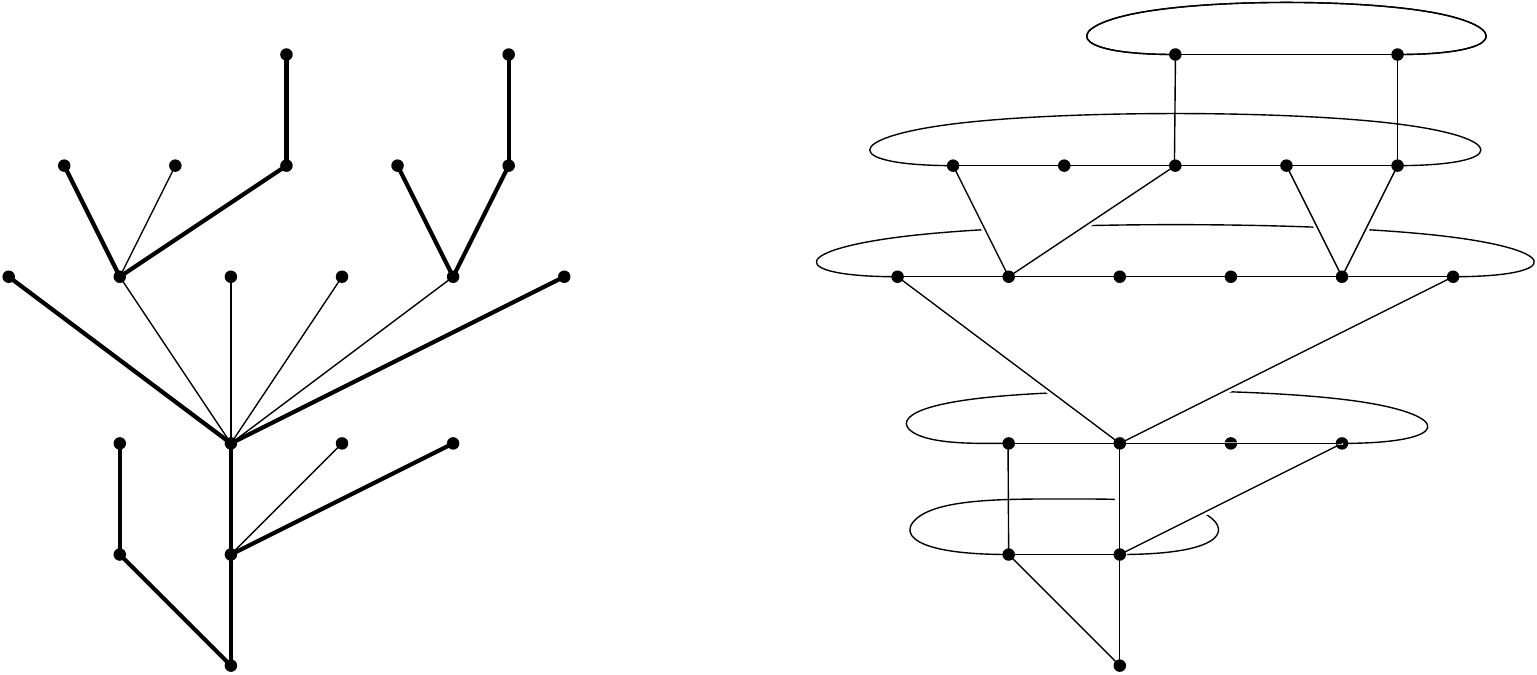}
 \caption{ \small{\label{kri1} The causal carpet is obtained from the causal map by deleting all but the extreme-most vertical edges emanating upwards from each vertex.
 }}
 \end{center}
 \vspace{-.75em}
 \end{figure}
 We call this graph the \textbf{causal carpet} associated to the tree. The geometry of the $\alpha$-stable causal carpet is very different from the maps studied in this work, since the faces of this map may now have very large degree.  In spite of this, the block-renormalisation methodology developed in Section \ref{sec:halfplane} carries through to this model, and analogs of Theorem \ref{thm:geometrystable}, as well as of the resistance exponent bound
\[
\mathfrak{r} \leq \frac{2-\alpha}{\alpha-1}
\]
  hold true. 
  Alas, this resistance bound becomes trivial precisely at the most interesting value of $\alpha=3/2$, for which a graph closely related to the causal carpet can be realized as a subgraph of the UIPT via Krikun's skeleton decomposition or via the recent construction of \cite{curien2018skeleton}.
  Indeed, it remains open to prove any sublinear resistance upper bound for the UIPT. Such a bound would (morally) follow from the $\alpha=3/2$ case of the following conjecture.

\begin{conjecture}
\label{conj:carpets}
Let $\mu$ be critical and satisfy \eqref{eq:defstable} for some $\alpha\in (1,2)$. Then the resistance growth exponent $\mathfrak{r}$ of the associated causal carpet exists and satisfies $0<\mathfrak{r}<1$ almost surely. In particular, the causal carpet is recurrent almost surely.
\end{conjecture}


Despite the sub-optimality of our spectral results in this context, the \emph{geometric} results obtained by our methods are sharp. The applications of our methodology to uniform random planar triangulations will be explored further in 
a future work. \medskip 

Finally, we remark that a model essentially equivalent to the uniform CDT arises as a certain $\gamma \downarrow 0$ limit of Liouville Quantum Gravity (LQG) in the mating-of-trees framework \cite{wedges,ghs-dist-exponent}. (More specifically, it arises as the $\gamma \downarrow 0$ limit of the mated-Galton-Watson-tree model of LQG, in which the correlation of the encoding random walks tends to $-1$). Thus, further study of the uniform CDT may prove useful for understanding LQG in the small $\gamma$ regime, which has recently been of great interest following Ding and Goswami's refutation of the Watabiki formula \cite{ding2016upper}.\medskip

\textbf{Added in proof.} Gwynne and Miller \cite{gwynne2017random} recently proved sub-polynomial resistance estimates for the UIPT. Conjecture \ref{conj:carpets} and the analogous question for the skeleton of the UIPT remain open, however.

\linespread{1}
\bibliographystyle{siam}
\bibliography{bibli}

\end{document}